\newif\ifcolt

\colttrue
\ifcolt
\documentclass{colt2019} % Anonymized submission
\usepackage{lbs_colt}
\usepackage{times}
\else
\documentclass[11pt]{article}
\newcommand{\includecomments}{1}

\usepackage{lbs}
\usepackage[numbers]{natbib}
\usepackage[top=2.54cm,left=3cm,right=3cm,bottom=2.54cm]{geometry}

\usepackage{xcolor}
\definecolor{darkblue}{rgb}{0,0,.3}
\usepackage[colorlinks=true,allcolors=darkblue]{hyperref}

\fi

\usepackage{additional-macros}

\renewcommand{\theassumption}{A\arabic{assumption}}

\begin{document}

\ifcolt
\title[Lower Bounds for Locally Private Estimation]{Lower
  Bounds for Locally Private Estimation via \\ Communication Complexity}
% \usepackage{times}
% Authors with different addresses:
\coltauthor{%
  \Name{John Duchi} \Email{\href{mailto:jduchi@stanford.edu}{jduchi@stanford.edu}}\\
  \addr Departments of Statistics and Electrical
  Engineering, Stanford University, and Apple
  \AND
  \Name{Ryan Rogers}
  \Email{\href{mailto:ryan_rogers@apple.com}{ryan\_rogers@apple.com}}\\
  \addr Apple%
}

\maketitle

\else
\begin{center}
  {\Large
    Lower Bounds for Locally Private Estimation via \\
    Communication Complexity
  } \\
  \vspace{.4cm}
  \begin{tabular}{cc}
    \large{John Duchi} & \large{Ryan Rogers} \\
    \large{Stanford University} &
    \large{Apple}
  \end{tabular}
\end{center}
%% \title{Lower Bounds for Locally Private Estimation via \\
%%   Communication Complexity}
%% \author{John Duchi \\ Stanford University
%%   \and Ryan Rogers \\ Apple}
\fi

\begin{abstract}
  We develop lower bounds for estimation under local privacy
  constraints---including differential privacy and its relaxations to
  approximate or
  R\'{e}nyi differential privacy---by showing an equivalence
  between private estimation and communication-restricted estimation
  problems. Our results apply to arbitrarily
  interactive privacy mechanisms, and they also give sharp lower bounds for
  all levels of differential privacy protections, that is, privacy
  mechanisms with privacy levels $\diffp \in [0, \infty)$.  As a particular
  consequence of our results, we show that the minimax mean-squared error for
  estimating the mean of a bounded or
  Gaussian random vector in $d$ dimensions
  scales as $\frac{d}{n} \cdot \frac{d}{ \min\{\diffp, \diffp^2\}}$.
\end{abstract}

% !TEX root = communication-lower-bounds.tex

% -*- Mode: latex -*- %

\section{Introduction}

Estimation problems in which users keep their personal data private
even from data collectors are of increasing interest in
large-scale machine learning applications in both
industrial~\citep{ErlingssonPiKo14, ApplePrivacy17, BhowmickDuFrKaRo18} and
academic~\citep[e.g.][]{Warner65, BeimelNiOm08, KasiviswanathanLeNiRaSm11, DuchiJoWa18} settings. These notions of privacy
are satisfying because a user or data provider can be confident that his or
her data will remain private irrespective of what data collectors do, and
they mitigate risks for data collectors,
limiting challenges of hacking or
other interference.  Because of their importance, a parallel literature on
optimality results in local privacy
is developing. % ~\citep{DuchiJoWa18, DuchiRu18b, YeBa18}.

Yet this theory fails to address a number of important issues. Most
saliently, many of these results only apply in settings in which the
privatization scheme is non-adaptive, that is, the scheme remains static for
all data contributors except in 1-dimensional problems~\citep{DuchiJoWa18,
  YeBa18, GaboardiRoSh18}.  A second issue is that these results provide
meaningful bounds only for certain types of privacy.  Typically, the results
are sharp only for high levels of privacy (in the language of differential
privacy, privacy parameters $\diffp \le 1$), as in the papers
of~\citet{DuchiJoWa18}, \citet{RohdeSt18}, and \citet{DuchiRu18b}, or at
most logarithmic in dimension~\citep{YeBa18}; given the promise of privacy
amplification in local settings~\citep{ErlingssonFeMiRaTaTh19} and
challenges of high-dimensional problems~\citep{DuchiJoWa18,DuchiRu18b}, it
is important to address limits in the case that $\diffp \gg 1$. With the
exception of \citeauthor{DuchiRu18b}, they also fail to apply to weakenings
of differential privacy.

We remove many of these restrictions by framing the problem of estimation
and learning under local privacy constraints as a problem in the
communication complexity of statistical estimation.  By doing so, we can
build off of a line of sophisticated results due to
\citet{ZhangDuJoWa13_nips}, \citet{GargMaNg14}, and
\citet{BravermanGaMaNgWo16}, who develop minimax lower bounds on distributed
estimation problems.  To set the stage for our results and give intuition
for what follows, we recall the intuitive consequences of these results.
Each applies in a setting in which $n$ machines each receive a sample
$X_i$ from an underlying (unknown) population distribution $P$. These
machines then interactively communicate with a central server, or a public
blackboard, sending $n \cdot I$ bits of (Shannon) information in total,
so that each sends an average of $I$
bits. For $d$-dimensional mean estimation problems, where $X_i \in \R^d$ and
the goal is to estimate $\E_P[X]$, the main consequences of
these papers is that the mean-squared error for estimation must scale as
$\frac{d}{n} \cdot \max\{\frac{d}{I}, 1\}$, where $d/n$ is the optimal
(communication unlimited) mean-squared error based on a sample of size $n$.
Such scaling is intuitive, as to estimate a $d$-dimensional quantity, we
expect each machine must send roughly $d$ bits to achieve optimal
complexity, and otherwise we receive information about only
$d/I$ coordinates.  The strength of these results is that, in the most general
case~\citep{BravermanGaMaNgWo16}, they allow essentially arbitrary
interaction between the machines, so long as it is mitigated by the
information constraints.

We leverage these results on information-limited estimation to
establish lower bounds for locally private estimation.  By providing bounds
on the information released by locally private protocols---even when data
release schemes are adaptive and arbitrarily interactive---we can nearly
immediately provide minimax lower bounds on rates of convergence in
estimation and learning problems under privacy. By using this
information-based-complexity framework, we can simultaneously address each
of the challenges we identify in previous work on estimation under privacy
constraints, in that our results apply to differential privacy and its
weakenings, including approximate, concentrated, and R\'{e}nyi differential
privacy~\citep{DworkMcNiSm06, DworkKeMcMiNa06, DworkRo16, BunSt16,
  Mironov17}.  They also apply to arbitrarily interactive data release
scenarios.  Roughly, what we show is that so long as we wish to estimate
quantities for $d$-dimensional parameters that are ``independent'' of one
another---which we define subsequently---the effective sample size available
to a private procedure reduces from $n$ to $n \cdot \min\{\diffp, \diffp^2,
d\} / d$ for all $\diffp$-private procedures.

The use of information and communication complexity in determining
fundamental limits in differential privacy is not uniquely ours.
\citet{McGregorMiPiReTaVa10} show strong relationships between
approximating functions by low-error
differentially private protocols and low communication protocols.
In their case, however, they study low error approximation
of \emph{sample} quantities, where one wishes to estimate
$f(X_1, \ldots, X_n)$ for a function $f$.  Here, as in most work in
statistics and learning~\citep{Wainwright19, Yu97, DuchiJoWa18},
we provide limits on the estimation
functions of the population from which the sample comes.
In recent work,
\citet[Sec.~5]{JosephKuMaWu18} give communication-based bounds for
locally-private estimation of a 1-dimensional Gaussian mean; their bound
requires a single pass through the data and privacy parameter $\diffp =
O(1)$.

As a consequence of our lower bounds, we identify several open questions.
Work in information-limited
estimation~\citep{ZhangDuJoWa13_nips, GargMaNg14, BravermanGaMaNgWo16}
typically strongly relies on independence among estimands, which allows
decoupling them.  Our results similarly suffer these
restrictions, which is essential: when correlations exist among different
coordinates of the sample vectors $X$, it is possible to achieve
faster convergence. Thus, we argue that we should have renewed focus on
\emph{local} (non-minimax) notions of complexity~\citep{LeCamYa00,
  VanDerVaart98, DuchiRu18b}, which address the difficulty of the particular
problem at hand.
%% In
%% classical statistics, the theory of local asymptotic normality and
%% minimaxity addresses these issues; a modern treatment of these in the face
%% of restricted estimators would be valuable.

\paragraph{Notation}
We index several quantities.  We always indicate
coordinates of a vector by $j$, and (independent) vectors we index
by $i$. We consider private protocols communicating in rounds
indexed by time $t$.  We let $Z_{\leq i} \defeq (Z_1, \ldots,
Z_i)$ and $Z_{< i} \defeq (Z_1, \ldots, Z_{i-1})$, and similarly for superscripts.
For distributions $P$ and $Q$, $\drenyi{P}{Q} \defeq
\frac{1}{\alpha - 1} \log \int (dP/dQ)^\alpha dQ$ is the R\'{e}nyi
$\alpha$-divergence.

% !TEX root = communication-lower-bounds.tex

% -*- Mode: latex -*- %

\section{Problem setting and main results}

We first describe our problem setting in detail, providing graphical
representations of our privacy (or communication) settings.
We present corollaries of our main lower bounds to highlight
their application, then (in Section~\ref{sec:lower-bounds})
give the main techniques, which extend Assouad's
method.

\subsection{Local privacy and interactivity}

In our local privacy setting, we consider $n$ individuals, each with private
data $X_i$, $i = 1, \ldots, n$, and each individual $i$ communicates
privatized views $Z_i$ of $X_i$.
%% In contrast to other work on lower bounds
%% for (local) differential privacy,
This private communication may depend on other data providers'
private data. We consider communication of privatized data
in rounds $t = 1, 2, \ldots, T$, where $T$ may be infinite, and in round
$t$, individual $i$ communicates private datum $Z_i^{(t)}$, which
may depend on all previous private communications. This is the standard
blackboard communication model; at round $t$ the $Z_i^{(t)}$ and
previous blackboards $B^{(t-1)}$ join into $B^{(t)} = (Z_{\leq n}^{(t)},
B^{(t-1)})$. Thus, at round $t$, individual $i$ generates the
private variable $Z^{(t)}_i$ according to the channel
\begin{equation*}
  \channel_{i,t}\big(\cdot \mid X_i, Z_{<i}^{(t)}, B^{(t-1)}\big).
\end{equation*}
Figure~\ref{fig:communication-scheme} illustrates this communication scheme
over two rounds of communication.
We require
that the channels be regular conditional
  probabilities~\citep{Billingsley86}.

\begin{figure}[t]
  \begin{center}
    \begin{tabular}{cc}
      \begin{minipage}{.7\columnwidth}
        \scalebox{.7}{
          \begin{overpic}[width=\columnwidth]{%,grid]{%
              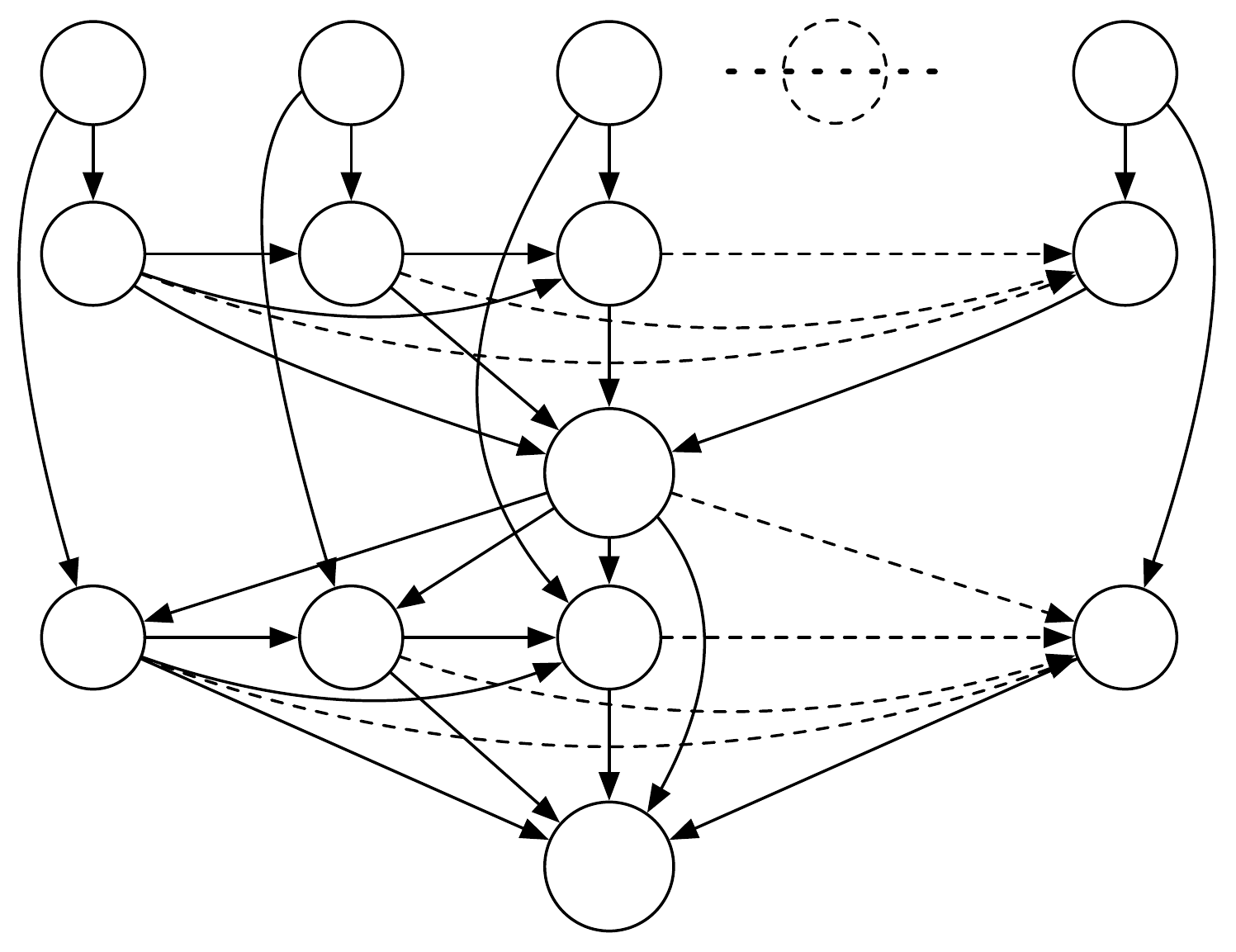}
            %% \put(4.5,69){$X^{(1)}$}
            %% \put(25.3,69){$X^{(2)}$}
            %% \put(46.5,69){$X^{(3)}$}
            %% \put(88,69){$X^{(n)}$}
            \put(4.9,70){$X_1$}
            \put(25.9,70){$X_2$}
            \put(47,70){$X_3$}
            \put(88.5,70){$X_n$}
            \put(4,55){$Z^{(1)}_1$}
            \put(25,55){$Z^{(1)}_2$}
            \put(46,55){$Z^{(1)}_3$}
            \put(88,55){$Z^{(1)}_n$}
            \put(47,37){$B^{(1)}$}
            \put(4,24){$Z^{(2)}_1$}
            \put(25,24){$Z^{(2)}_2$}
            \put(46,24){$Z^{(2)}_3$}
            \put(88,24){$Z^{(2)}_n$}
            \put(47,5){$B^{(2)}$}
          \end{overpic}
        }
      \end{minipage}
      &
      \hspace{-3cm}
      \begin{minipage}{.4\columnwidth}
        \caption{\label{fig:communication-scheme}
          Two rounds of communication of variables, writing to public
          blackboards $B^{(1)}$ and $B^{(2)}$.}
      \end{minipage}
    \end{tabular}
  \end{center}
  \vspace{-.6cm}
\end{figure}
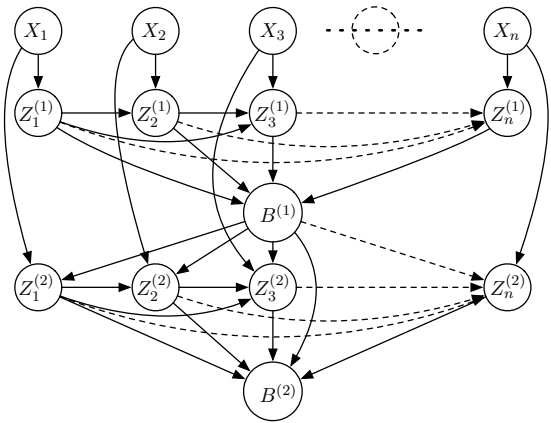

Our main assumptions are that the channels satisfy
quantitative privacy definitions.
\begin{definition}
  \label{definition:differential-privacy}
  Let $\diffp \ge 0$. A random variable $Z$ is
  \emph{$(\diffp,\delta)$-differentially
    private}~\citep{DworkMcNiSm06,DworkKeMcMiNa06} for $X \in \mc{X}$ if
  conditional on $X = x$, $Z$ has distribution $\channel(\cdot \mid x)$ and
  for all measurable $S$ and $x, x'$,
  \begin{equation*}
    \channel(Z \in S \mid x) \le e^{\diffp} \channel(Z \in S \mid x')
    + \delta.
  \end{equation*}
  When $\delta = 0$, we say $\channel$ is \emph{$\diffp$-differentially
    private}.  For $\renparam \ge 1$, the channel is \emph{$(\diffp,
    \renparam)$-R\'{e}nyi differentially private}~\citep{Mironov17} if for
  all $x, x' \in \mc{X}$,
  \begin{equation*}
    \drenyi{\channel(\cdot \mid x)}{\channel(\cdot \mid x')}
    \le \diffp.
  \end{equation*}
\end{definition}
\noindent
By taking $\renparam = 1$ in
Definition~\ref{definition:differential-privacy}, we obtain
\emph{$\diffp$-KL-privacy}.
%% meaning that for any $x, x'$, the channel
%% $\channel$ satisfies $\dkl{\channel(\cdot \mid x)}{\channel(\cdot \mid x')}
%% \le \diffp$.
If the channel $\channel$ is $\diffp$-differentially private,
then for any $\renparam \ge 1$, it also satisfies~\cite[Lemma 1]{Mironov17}
\begin{equation}
  \label{eqn:diffp-to-renyi}
  \drenyi{\channel(\cdot \mid x)}{\channel(\cdot \mid x')}
  \le
  \min\left\{2 (\renparam - 1) \diffp^2
  + \min\{2, (e^\diffp - 1)\} \diffp, \diffp \right\}.
  %% \min\left\{
  %% \begin{array}{l}
  %%   2 (\renparam - 1) \diffp^2
  %%   + \min\{2, (e^\diffp - 1)\} \diffp
  %%   \le 2 \renparam \diffp^2, \\
  %%   \diffp
  %% \end{array}
  %% \right\}.
\end{equation}
Because R\'{e}nyi-divergence is non-decreasing in $\renparam$, any $(\diffp,
\renparam)$-R\'{e}nyi differentially private channel is also $(\diffp,
\renparam')$-R\'{e}nyi private for $\renparam' \le \renparam$, making
KL-privacy the weakest R\'{e}nyi privacy.

We consider channel and disclosure scenarios where users
and data providers obtain a given amount of
privacy, but multiple notions of privacy
are possible. We separate these by allowing either
full interactivity or requiring a type of compositionality
of the private data releases; for
more on this distinction and examples separating the classes, see
\citet{JosephMaNeRo19}.

\subsubsection{Fully interactive privacy mechanisms}
\label{sec:fully-interactive}

The first and weakest assumptions on privacy we make are that the
private $\bZ \defeq \{Z_i^{(t)}\}_{i,t}$, or the entire communication
transcript, is private. To define this locally private setting, we require
an appropriate definition of privacy, for which we use
\citet{FeldmanSt18}.
%% , which in turn
%% extends~\citet[Def.~4.2]{BassilyNiSmStStUl16}.
\begin{definition}
  \label{definition:kl-stability}
  Let $\channel(\bZ \in \cdot \mid x_{\le n})$ denote the
  distribution of the collection $\bZ$ conditional on
  $X_{\le n} = x_{\le n}$, and for $i = 1, \ldots, n$, let
  the samples $x_{\le n}$ and $x_{\le n}^{(i)} \in \mc{X}^n$ differ
  in only example $i$, otherwise being arbitrary. The output $\bZ$ is
  \emph{$\diffpkl$-KL-locally private on average} if
  \begin{equation*}
    \frac{1}{n} \sum_{i = 1}^n
    \dkl{\channel(\bZ \in \cdot \mid x_{\le n})}{
      \channel(\bZ \in \cdot \mid x_{\le n}^{(i)})} \le \diffpkl.
  \end{equation*}
\end{definition}
\noindent
Definition~\ref{definition:kl-stability} is weaker than most versions of
local privacy. The most general standard notion of R\'{e}nyi $(\diffp,
\renparam)$-privacy is that $\drenyi{\channel(\bZ \in \cdot \mid x_{\le
    n})}{\channel(\bZ \in \cdot \mid x_{\le n}'} \le \diffp$ for all samples
$x_{\le n}$ and $x_{\le n}'$ differing in a single entry; this
immediately implies Definition~\ref{definition:kl-stability}.
We thus make the following
\begin{assumption}[Fully interactive local differential privacy]
  \label{assumption:general-epsilon-privacy}
  The entire output collection
  $\bZ$ is $\diffpkl$-KL-locally private on average
  (Definition~\ref{definition:kl-stability}).
  %% For each pair of samples $x_{\le
  %%   n}, x'_{\le n} \in \mc{X}^n$ differing in at most a single element, the
  %% entire output $\bZ$ is $\diffp$-differentially private, that is, for each
  %% set $S \in \mc{Z}^{nT}$,
  %% \begin{equation*}
  %%   \channel\left(\bZ \in S \mid X_{\le n} = x_{\le n}\right)
  %%   \le
  %%   e^\diffp \channel\left(\bZ \in S \mid X_{\le n} = x'_{\le n}\right).
  %% \end{equation*}
\end{assumption}
\noindent
Assumption~\ref{assumption:general-epsilon-privacy} makes no assumptions on
the local randomizers, requiring that the entire set of communicated private
views $\bZ$ is private for each individual $i$. There may be challenges in
the implementation of general protocols satisfying
Assumption~\ref{assumption:general-epsilon-privacy} if the privacy of user
$i$ depends on the behavior of user $i'$---adversarial
users---though for the purposes of lower bounds, this appears to be the
weakest model of local privacy.
Assumption~\ref{assumption:general-epsilon-privacy} is also weaker than the
assumption that the private variables $\bZ$ are $\diffp$-differentially private:
inequality~\eqref{eqn:diffp-to-renyi} shows that $\diffp$-differential
privacy implies $\diffpkl = \min\{\diffp, \diffp^2 / \log 2\}$-KL
privacy. Thus, if each individual $i$ is guaranteed (differential) privacy
loss $\diffp_i$, the KL-privacy loss satisfies
\begin{equation}
  \diffpkl \le \frac{1}{n}
  \sum_{i = 1}^n \min\left\{\diffp_i, \frac{\diffp_i^2}{\log 2}\right\}.
  \label{eqn:diffp-to-average-diffp}
\end{equation}

We can also consider (fully interactive) local approximate differential
privacy, though in the case that $\delta > 0$, our lower bounds require a
slight technical modification of
Assumption~\ref{assumption:general-epsilon-privacy}, requiring that the
domain $\mc{X}$ of the data $X_i$ be finite and $\delta$ be appropriately
small.  \setcounter{assumption}{0}
\renewcommand{\theassumption}{A\arabic{assumption}$^\prime$}
\begin{assumption}[Fully interactive local approximate differential privacy]
  \label{assumption:general-epsilon-delta-privacy}
  The output $\bZ$ is $(\diffp,
  \delta)$-differentially private: for each $S \subset
  \mc{Z}^{nT}$
  and pair of samples $x_{\le n}, x'_{\le n} \in \mc{X}^n$ differing in
  at most a single element,
  \begin{equation*}
    \channel\left(\bZ \in S \mid X_{\le n} = x_{\le n}\right)
    \le
    e^\diffp \channel\left(\bZ \in S \mid X_{\le n} = x'_{\le n}\right)
    + \delta.
  \end{equation*}
  In addition, the parameters $(\diffp, \delta)$ satisfy
  \begin{equation*}
    \delta \le \frac{\min\{\diffp, 1\}}{256},
    ~~~
    \delta \max\{\diffp^{-1}, 1\}
    \log \frac{1}{\delta \max\{\diffp^{-1}, 1\}} \le \diffp^2,
    ~~~
    \delta \le \frac{\min\{\diffp, \diffp^2\}}{\log^2 |\mc{X}|}
  \end{equation*}
  and, if $\diffp \le \frac{1}{6}$, also
  $\delta \le \frac{\diffp^5}{64 \log^2 |\mc{X}|}$
  and $\delta \log^2 \frac{\diffp}{\delta}
  \le \diffp^5 / 16$.
\end{assumption}
\renewcommand{\theassumption}{A\arabic{assumption}}

\subsubsection{Compositional local privacy mechanisms}
\label{sec:compositional}

A different modification of
Assumption~\ref{assumption:general-epsilon-privacy} is to require the
individual randomizations be private while imposing
a summability (compositionality) condition.
This limits the interaction between private communications, and
there are problems for which fully interactive mechanisms are more powerful
than compositional ones~\citep{JosephMaNeRo19}.
To be concrete, let
$\intoit{Z} \defeq (Z_{< i}^{(t)}, B^{(t-1)})$
%% \begin{equation*}
%%   \intoit{Z}
%%   \defeq \left(Z_{<i}^{(t)}, B^{(t-1)}\right)
%%   = \left(Z_1^{(1)}, \ldots, Z_n^{(1)}, Z_1^{(2)}, \ldots, Z_n^{(2)},
%%   \ldots, Z_1^{(t)}, \ldots, Z_{i-1}^{(t)}\right)
%% \end{equation*}
be the ``messages'' coming into the channel generating $Z_i^{(t)}$, so
$Z_i^{(t)} \sim \channel_{i,t}(\cdot \mid X_i, \intoit{Z})$ as in
Fig.~\ref{fig:communication-scheme}.
%% Each $\channel$ is a regular
%% conditional probability, meaning $\channel(\cdot \mid x, \intoit{z})$ is a
%% probability distribution for each $x, \intoit{z}$, and for each measurable
%% set $S$, $(x, \intoit{z}) \mapsto \channel(S \mid x, \intoit{z})$ is a
%% measurable function.
The starting point is
a \emph{1-compositional} privacy
definition~\citep{JosephMaNeRo19}, where in the (weakest) KL-privacy case we
assume that there exists a function $\diffp_{i,t}$ such that
\begin{equation}
  \label{eqn:summed-kl-privacy}
  \dkl{\channel_{i,t}(\cdot \mid X_i = x, \intoit{Z} = \intoit{z})}{
      \channel_{i,t}(\cdot \mid X_i = x',
      \intoit{Z} = \intoit{z})}
  \le \diffp_{i,t}(\intoit{z})
\end{equation}
and the $\diffp_{i,t}$ satisfy $n^{-1}\sum_{i =
  1}^n \sum_{t = 1}^T \E[\diffp_{i,t}(\intoit{Z}) \mid x_{\le n}] \le
\diffpkl$. Condition~\eqref{eqn:summed-kl-privacy} implies
Assumption~\ref{assumption:general-epsilon-privacy}
by the chain rule for KL-divergence;
%% : by the chain-rule for
%% KL-divergence~\citep{CoverTh06}, we have for any $i_0 \in [n]$ that
%% \begin{align*}
%%   \lefteqn{\dkl{\channel(\bZ \in \cdot \mid x_{\le n})}{
%%     \channel(\bZ \in \cdot \mid x_{\le n}^{(i_0)})}} \\
%%   & = \sum_{i = 1}^n \sum_{t = 1}^T
%%   \E\left[\dkl{\channel_{i,t}(\cdot \mid x_{\le n}, \intoit{Z})}{
%%       \channel_{i,t}(\cdot \mid x_{\le n}^{(i_0)}, \intoit{Z})}
%%     \mid x_{\le n}\right] \\
%%   & \stackrel{(\star)}{=} \sum_{t = 1}^T
%%   \E\left[\dkl{\channel_{i_0,t}(\cdot \mid x_{i_0}, \intoitalt{Z}{i_0})}{
%%       \channel_{i_0,t}(\cdot \mid x'_{i_0}, \intoitalt{Z}{i_0})}
%%     \mid x_{\le n}\right]
%%   \le \sum_{t = 1}^T \E[\diffp_{i_0,t}(\intoitalt{Z}{i_0}) \mid x_{\le n}],
%% \end{align*}
%% where equality~$(\star)$ follows because given $\intoitalt{Z}{i_0}$,
%% $Z_{i_0,t}$ depends only on $X_{i_0}$. Summing over indices $i_0$ shows
%% that condition~\eqref{eqn:summed-kl-privacy} implies
%% Assumption~\ref{assumption:general-epsilon-privacy}.
in the case that $T =
1$, this captures the familiar \emph{sequentially
  interactive} local privacy mechanisms~\citep{DuchiJoWa18}.

In some cases, we can provide stronger results for compositional
$(\diffp,\delta)$-private channels than for the fully interactive case,
leading us to consider the following assumption,
which allows privacy levels
chosen conditionally on the past so long as the expected privacy
levels remain non-trivial.
\begin{assumption}[Compositional differential privacy bounds]
  \label{assumption:summed-approximate-privacy}
  For each $i$ and $t$ and all $\intoit{z}$,
  the channel
  mapping $X_i$ to $Z_i^{(t)}$ is $(\diffp_{i,t}(\intoit{z}),
  \delta_{i,t}(\intoit{z}))$-approximately differentially private.  There
  exist $\totaldelta \le \half$ and $\diffpkl$ such that
  \begin{equation*}
    \sum_{i = 1}^n \sum_{t = 1}^T \E[\delta_{i,t}(\intoit{Z})]
    \le \totaldelta
    ~~ \mbox{and} ~~
    \sum_{i = 1}^n \sum_{t = 1}^T \E\left[\min\left\{\frac{\diffp_{i,t}^2(
      \intoit{Z})}{\log 2}, \diffp_{i,t}(\intoit{Z})\right\}\right]
    \le n \cdot \diffpkl,
  \end{equation*}
  %% Additionally, the $\delta_{i,t}$ satisfy
  %% \begin{equation*}
  %%   \delta_{i,t} \le \frac{\min\{\diffp_{i,t}, 1\}}{64},
  %%   ~~
  %%   \delta_{i,t} \le \frac{\diffp_{i,t}^2}{\log^2 |\mc{X}|},
  %%   ~~
  %%   \delta_{i,t} \le \frac{\diffp_{it}^5}{64 \log^2 |\mc{X}|},
  %%   ~~
  %%   \frac{\delta_{i,t}}{\diffp_{i,t}}
  %%   \log^2 \frac{\diffp_{i,t}}{\delta_{i,t}}
  %%   \le \frac{\diffp_{i,t}^4}{16}.
  %% \end{equation*}
  %% Finally, there exists $\diffpkl < \infty$ such that
  %% for all $T \in \N$,
  %% \begin{equation*}
  %%   8 \sum_{i = 1}^n \sum_{t = 1}^T
  %%   \E\left[\min\left\{10 \cdot \diffp_{i,t}^2(\intoit{Z}),
  %%     \diffp_{i,t}(\intoit{Z})\right\}\right]
  %%   \le n \cdot \diffpkl,
  %% \end{equation*}
  where the expectations are taken over the randomness in the private
  variables $\bZ$.
\end{assumption}
\noindent
In  Assumption~\ref{assumption:summed-approximate-privacy},
individual $i$ compromises
at most $(\sum_t \diffp_{i,t}, \sum_t \delta_{i,t})$-differential privacy.

\subsection{Minimax lower bounds on private estimation\label{sect:minimax}}

Given our definitions of (interactive) privacy and the interactive privacy
bounds in Assumptions~\ref{assumption:general-epsilon-privacy},
\ref{assumption:general-epsilon-delta-privacy},
and~\ref{assumption:summed-approximate-privacy}, we may now describe the
minimax framework in which we work. Let $\mc{P}$ be a collection of
distributions on a space $\mc{X}$, and let $\theta(P) \in \Theta$ be a
parameter of interest.  In
the classical (non-information-limited) setting, we wish to
estimate $\theta(P)$ given observations $X_i$ drawn i.i.d.\ according to the
distribution $P$.  We focus on $d$-dimensional parameters
$\theta$, and the performance of an estimator $\what{\theta} : \mc{X}^n \to
\R^d$ is its expected loss (or risk) for a loss $L : \R^d \times \R^d \to
\R_+$,
\begin{equation*}
  \E_P\left[ L(\what{\theta}(X_1, \ldots, X_n), \theta(P))
    \right].
\end{equation*}

We elaborate this classical setting by an additional privacy layer.
For a sample $\{ X_{1}, \ldots, X_n\}$, any
(interactive) channel $\channel$ produces a set of private
observations, each from some set $\mc{Z}$, 
\begin{equation*}
  \bZ \defeq \left(Z_1^{(1)}, Z_2^{(1)}, \ldots, Z_n^{(1)},
Z_1^{(2)}, \ldots, Z_n^{(2)}, \ldots, Z_n^{(T)} \right) \in \mc{Z}^{T \times
  n},
\end{equation*}
and we consider estimators $\what{\theta}$ that depend only on this
private sample, which then suffer risk
\begin{equation*}
  \E_{P,Q}\left[L(\what{\theta}(\bZ),
    \theta(P))\right],
\end{equation*}
where the expectation is taken over the $n$ i.i.d.\ observations
$X_i \sim P$ and the privatized views $\bZ$.
For the channel $\channel$,
we define the \emph{channel minimax risk} for the family
$\mc{P}$, parameter $\theta$, and loss $L$ by
\begin{equation}
  \label{eqn:minimax-risk}
  \minimax_n(\theta(\mc{P}), L, \channel)
  \defeq \inf_{\what{\theta}}
  \sup_{P \in \mc{P}}
  \E_{P,Q}\left[L\left(\what{\theta}(\bZ),
    \theta(P)\right)\right].
\end{equation}
We prove lower bounds on the quantity~\eqref{eqn:minimax-risk}
for  channels satisfying local privacy
bounds.

Rather than stating and proving our main theorems,
we present a number of corollaries of our
main results, all of whose proofs we defer to
\ifcolt
Appendix~\ref{sec:proofs-corollaries},
\else
Section~\ref{sec:proofs-corollaries},
\fi
to illustrate the power of the information-based framework we adopt.
Our first corollary deals with estimating Bernoulli
means.
\begin{corollary}
  \label{corollary:bernoulli}
  Let $\mc{P}_d$ be the collection of Bernoulli distributions on $\{0,
  1\}^d$ and  for a symmetric loss $\ell : \R \to \R_+$ minimized at
  $0$,
  let $L(\theta, \theta') = \sum_{j = 1}^d \ell(\theta_j -
  \theta_j')$. There are numerical constants $c_1,c_2,c_3 > 0$ such that for any
  channel $\channel$ satisfying any of
  Assumptions~\ref{assumption:general-epsilon-privacy},
  \ref{assumption:general-epsilon-delta-privacy} with $\diffpkl \defeq
  \min\{\diffp, \diffp^2\}$, or
  Assumption~\ref{assumption:summed-approximate-privacy} with privacy
  budget $\diffpkl$,
  \begin{equation*}
    \minimax_n(\theta(\mc{P}_d), L, \channel)
    \ge c_1 \cdot d
    \cdot \ell\left(\sqrt{c_2  \frac{d}{n \diffpkl}} \wedge c_3 \right).
    %% \minimax_n(\theta(\mc{P}_d), L, (\diffp, \delta))
    %% & \ge c_1 \cdot d
    %% \cdot \ell\left(\sqrt{\frac{c_2 d}{n \diffp}} \wedge c_3 \right).
  \end{equation*}
\end{corollary}
\noindent
In particular, if $\ell(t) = t^2$ and
if the private data releases of each individual
are $\diffp$-locally differentially private (under any model
of interaction), then
inequality~\eqref{eqn:diffp-to-renyi} and the corollary
imply
that for a constant $c>0$, for
any estimator $\what{\theta}$ there exists
a Bernoulli distribution $P$ with mean $\theta$ such that
\begin{equation*}
  % \label{eqn:diffp-bernoulli-lower}
  \E_{P,Q}\left[\ltwobig{\what{\theta}(\bZ) - \theta}^2\right]
  \ge c \left(\frac{d^2}{n \min\{\diffp, \diffp^2\}}
  \vee \frac{d}{n}\right).
\end{equation*}

%% In particular, if $\ell(t) = t^2$, then
%% for any channel $\channel$ satisfying
%% Assumption~\ref{assumption:general-epsilon-privacy}
%% or Assumption~\ref{assumption:summed-approximate-privacy} and any
%% estimator $\what{\theta}$, there exists
%% a Bernoulli distribution $P$ with mean $\theta$ such that
%% \begin{equation}
%%   \label{eqn:bernoulli-ltwo-lower}
%%   \E_{P,Q}\left[\ltwobig{\what{\theta}(\bZ) - \theta}^2\right]
%%   \ge c \left(\frac{d^2}{n \diffpkl} \vee \frac{d}{n}\right).
%% \end{equation}
%% A second consequence of this result is that, if
%% the private data releases of each individual are $\diffp$-differentially
%% private, then
%% inequalities~\eqref{eqn:diffp-to-renyi}
%% and~\eqref{eqn:diffp-to-average-diffp} imply that for a constant $c>0$
%% \begin{equation*}
%%   % \label{eqn:diffp-bernoulli-lower}
%%   \E_{P,Q}\left[\ltwobig{\what{\theta}(\bZ) - \theta}^2\right]
%%   \ge c \left(\frac{d^2}{n \min\{\diffp, \diffp^2\}}
%%   \vee \frac{d}{n}\right).
%% \end{equation*}

A counterpart to the lower
bound of Corollary~\ref{corollary:bernoulli} is that
$\diffp$-differentially-private channels achieve this risk when $1 \le
\diffp \le d$, and they require no interactivity. To within numerical
constant factors, weakenings of local differential privacy---down
to KL-privacy---provide no rate of convergence improvement over differentially
private mechanisms.
\citet[Sec.~4.1]{BhowmickDuFrKaRo18} exhibit a mechanism
(\texttt{PrivUnit}$_2$), based on sampling from spherical
caps, that given $x$ satisfying $\ltwo{x} \le r$ samples
$\diffp$-differentially private $Z \in \R^d$ satisfying $\E[Z \mid x] = x$ and
$\ltwo{Z} \le C r \sqrt{d / \min\{\diffp, \diffp^2\}}$ for a numerical
constant $C$.
Taking the radius $r = \sqrt{d}$
%% (as $X_i \in \{0, 1\}^d$
%% satisfies $\ltwo{X_i} \le \sqrt{d}$),
the  estimator $\what{\theta}_n =
\frac{1}{n} \sum_{i = 1}^n Z_i$ satisfies
\begin{equation*}
  \E[\ltwos{\what{\theta}_n - \theta}^2]
  \le \frac{1}{n} \E[\ltwo{Z_1}^2]
  \le C \frac{d^2}{n \min\{\diffp, \diffp^2\}}.
\end{equation*}
For the simpler case of KL-privacy, Gaussian noise addition suffices.
We have thus characterized the complexity of locally
private $d$-dimensional estimation of bounded vectors.
%% For the slightly simpler case of KL-privacy, a mechanism
%% that simply sets $Z_i = X_i + W_i$ for
%% $W_i \simiid \normal(0, \frac{d}{2 \diffp} I_{d \times d})$ satisfies
%% $\dkl{\channel(\cdot \mid x)}{\channel(\cdot \mid x')}
%% = \frac{\diffp}{d} \ltwo{x - x'}^2 \le \diffp$, and
%% then $\what{\theta}_n = \frac{1}{n} \sum_{i = 1}^n Z_i$
%% satisfies $\E[\ltwos{\what{\theta}_n - \theta}^2]
%% \le \frac{d^2}{2 n \diffp} + \frac{d}{4n}$. We have
%% thus characterized the complexity
%% of locally private $d$-dimensional estimation of bounded vectors.

By a reduction, the lower bound of
Corollary~\ref{corollary:bernoulli} applies to logistic regression.
In this case, we let $\mc{P}_d$ be the collection of
logistic distributions on $(X, Y) \in \{-1, 1\}^d \times \{\pm 1\}$,
where for $\theta \in \R^d$,
$P(Y = y \mid X = x) = 1 / (1 + \exp(-y \<x, \theta\>))$.
We take the loss $L$ as the gap in prediction
risk: for
\begin{equation*}
  \loss(\theta; (x, y))
  = \log(1 + \exp(-y \<x, \theta\>)),
  ~~ \mbox{we set} ~~
  \risk_P(\theta) \defeq \E_P[\loss(\theta; (X, Y))]
\end{equation*}
and $\theta(P) = \argmin_\theta \risk_P(\theta)$. We define the
excess risk $L(\theta,
\theta(P)) = \risk_P(\theta) - \risk_P(\theta(P))$.
\begin{corollary}
  \label{corollary:logistic-lower-bound}
  Let $\mc{P}_d$ be the family of logistic distributions and $L$ be the
  excess logistic risk as above.  There exists a numerical constant $c > 0$
  such that for any sequence $\channel_n$ of channels satisfying any of
  Assumption~\ref{assumption:general-epsilon-privacy},
  or~\ref{assumption:general-epsilon-delta-privacy} with $\diffpkl =
  \min\{\diffp, \diffp^2\}$,
  or Assumption~\ref{assumption:summed-approximate-privacy}, for all
  suitably large $n$ we have
  \begin{equation*}
    \minimax_n(\theta(\mc{P}_d), L, \channel_n)
    \ge c \cdot \frac{d}{n} \cdot \frac{d}{\diffpkl}.
  \end{equation*}
\end{corollary}
\noindent
%% The proof of the corollary is a essentially follows because
%% it is possible to reduce the problem
%% to one of estimating a Bernoulli with independent coordinates.

It is also of interest to consider continuous distributions. For
concreteness, we consider estimation of general and sparse Gaussian means,
showing results that follow as corollaries of our information bounds and
\citet{BravermanGaMaNgWo16}.  We prove the lower bounds for channels
satisfying Assumption~\ref{assumption:general-epsilon-privacy}
or~\ref{assumption:summed-approximate-privacy}; proving them under
Assumption~\ref{assumption:general-epsilon-delta-privacy} remains a
challenge.\footnote{We use mutual information-based bounds, and on the
  (negligible) $\totaldelta$-probability event of a privacy failure under
  Assumption~\ref{assumption:general-epsilon-delta-privacy}, it is possible
  to release infinite information. For compositional channels satisfying
  Assumption~\ref{assumption:summed-approximate-privacy}, we show (see
  Lemma~\ref{lemma:tv-eps-delta} in
  Sec.~\ref{sec:information-compositional-bounds}) that each channel is
  within $\delta_{i,t}$-variation distance to a differentially private
  $(\delta_{i,t} = 0)$ channel, so lower bounds based on testing apply.  The
  argument fails in the fully interactive setting, because the interaction
  may break the independence structure of the communication upon which our
  results rely.}
\begin{corollary}
  \label{corollary:gaussian}
  Let $\mc{P}$ be the collection of Gaussian distributions $\normal(\theta,
  \sigma^2 I)$ where $\theta \in [-1, 1]^d$, $\sigma^2 > 0$ is known,
  and consider the squared
  $\ell_2$ loss $L(\theta, \theta') = \ltwo{\theta - \theta'}^2$.
  There exist numerical
  constants $c, c_0 > 0$ such that if the channel $\channel$
  satisfies
  Assumption~\ref{assumption:general-epsilon-privacy}
  or \ref{assumption:summed-approximate-privacy}
  with $\totaldelta \le c_0$,
  \begin{equation*}
    \minimax_n(\theta(\mc{P}), \ltwo{\cdot}^2, \channel)
    \ge c \cdot \min\left\{d,
    \max\left\{\frac{d}{\diffpkl} \cdot \frac{d \sigma^2}{n},
    \frac{d \sigma^2}{n}\right\}\right\}.
  \end{equation*}
\end{corollary}
\noindent
We
demonstrate how to achieve this risk in
Section~\ref{sec:generic-achievability}, showing
(as is the case for our other results) that it is
achievable by differentially private schemes.

We can also state lower bounds for the sparse case, using
\citet[Theorem~4.5]{BravermanGaMaNgWo16}.
Let $\mc{N}^d_{k,\sigma^2}$ denote the collection
of $k$-sparse Gaussian distributions $\normal(\theta, \sigma^2 I)$,
$\theta \in [-1, 1]^d$.
%%  and $\lzero{\theta} \le k$; we assume
%% $\sigma^2 > 0$ is known.
\begin{corollary}
  \label{corollary:sparse-gaussian}
  There exist numerical constants $c, c_0 > 0$ such that for any channel
  $\channel$ satisfying
  Assumptions~\ref{assumption:general-epsilon-privacy} or
  \ref{assumption:summed-approximate-privacy} with $\totaldelta
  \le c_0$, and $d \ge 2k$,
  \begin{equation*}
    \minimax_n(\theta(\mc{N}^d_{k,\sigma^2}), \ltwo{\cdot}^2, \channel)
    \ge c \min\left\{k,
    \max\left\{\frac{d}{\diffpkl}
    \cdot \frac{k \sigma^2}{n},
    \frac{k \sigma^2 \log \frac{d}{k}}{n} \right\}\right\}.
  \end{equation*}
\end{corollary}

%% The simplicity of the proofs of both of these results shows the power of the
%% information-based-complexity view of estimation under privacy
%% constraints. By building off of the results of the
%% papers~\cite{ZhangDuJoWa13_nips,GargMaNg14,BravermanGaMaNgWo16}, we can
%% almost immediately provide \emph{minimax} lower bounds on all locally
%% differentially private procedures, for all privacy levels, and with
%% arbitrary interactivity.

% !TEX root = communication-lower-bounds.tex

% -*- mode: latex -*- %

\newcommand{\bmu}{\pmb{\mu}}
\newcommand{\proj}{\Pi}

\section{Achievability, information complexity,
  independence, and correlation}
\label{sec:achievability-and-correlation}

The lower bounds in our corollaries
are achievable---we demonstrate each of these
here---but we highlight a more subtle question regarding correlation.
Each of our lower bounds relies on the independence
structure of the data: roughly, all the communication-based
bounds we discuss require the coordinates of $X$ to follow
a product distribution. The lower bounds in this case are intuitive:
we must estimate $d$-dimensional quantities using (on average) $\diffp$
bits, so we  expect penalties scaling as $d / \diffp$ because one
coordinate carries no information about the others. In cases where there is
correlation, however, we might hope for more
efficient estimation; we view this as a major open question in privacy and,
more broadly, information-constrained estimators.  To that end,
we briefly show (Section~\ref{sec:generic-achievability}) that each
of our lower bounds in
Corollaries~\ref{corollary:bernoulli}--\ref{corollary:gaussian} is
achievable. After this, we mention asymptotic results
for sparse estimation (Sec.~\ref{sec:sparse-estimation})
and correlated data problems (Sec.~\ref{sec:correlated-data}).

\subsection{Achievability by differentially-private estimators}
\label{sec:generic-achievability}

We first demonstrate that the results in each of our corollaries are
achievable by $\diffp$-differentially private channels with limited
interactivity. We have already done so for
Corollary~\ref{corollary:bernoulli}. For
Corollary~\ref{corollary:logistic-lower-bound}, Corollary~3.2 of
\citet{BhowmickDuFrKaRo18} gives the achievability result. We
provide the Gaussian results
for the sake of completeness. (For the one dimensional case, see
also~\cite{JosephKuMaWu18}.)

We begin by demonstrating a one-dimensional Gaussian estimator.
Let $X_i \simiid \normal(\theta, \sigma^2)$, where $\sigma^2$ is known
and $\theta \in [-1, 1]$.
Consider $\diffp$-differentially private version of $X_i$ defined by
\begin{equation}
  B_i \defeq \sign(X_i)
  ~~ \mbox{and} ~~
  Z_i = \frac{e^\diffp + 1}{e^\diffp - 1}
  \cdot \begin{cases} B_i & \mbox{w.p.}~ \frac{e^\diffp}{e^\diffp + 1} \\
    -B_i & \mbox{otherwise}.
  \end{cases}
  \label{eqn:sign-flip-private-gauss}
\end{equation}
Then $\E[Z_i
  \mid X_i] = \sign(X_i)$, and for $\Phi(t) = \P(\normal(0, 1) \le t)$
the standard Gaussian CDF,
we have
$\E_\theta[Z_i] = 1 - 2 \Phi(-\theta / \sigma)$.
%% \begin{equation*}
%%   \E_\theta[Z_i] = P_\theta(X_i \ge 0)
%%   - P_\theta(X_i \le 0)
%%   %% = \P(\normal(\theta, \sigma^2) \ge 0) - \P(\normal(\theta, \sigma^2) \le 0)
%%   %% = \P(\normal(0, \sigma^2) \ge -\theta) - \P(\normal(0, \sigma^2) \le -\theta)
%%   = 1 - 2 \Phi(-\theta / \sigma).
%% \end{equation*}
Letting $\wb{Z}_n = \frac{1}{n} \sum_{i = 1}^n Z_i$ be the average of
the $Z_i$, the estimator defined by solving $\wb{Z}_n = 1 - 2
\Phi(\what{\theta}_n / \sigma)$ is nearly unbiased. Projecting this
quantity onto $[-1, 1]$ gives the estimator
\begin{equation}
  \what{\theta}_n \defeq
  \mbox{Proj}_{[-1,1]}
  \left(\sigma \Phi^{-1}\left(\frac{1 - \wb{Z}_n}{2}\right)\right).
  \label{eqn:gaussian-estimator}
\end{equation}
This estimator satisfies the following, which
we prove in Appendix~\ref{sec:proof-gaussian-achievability}
via a Taylor expansion.
\begin{lemma}
  \label{lemma:gaussian-achievability}
  Let $\what{\theta}_n$ be the estimator~\eqref{eqn:gaussian-estimator} for
  the $\normal(\theta, \sigma^2)$ location family, where $\sigma^2 > 0$ is
  at least a constant.  Assume $|Z_i| \le b$ and $\E[Z_i] = 1 -
  2 \Phi(-\theta / \sigma)$.  For numerical constants $0 < c \le C <
  \infty$,
  \begin{equation*}
    |\what{\theta}_n - \theta|
    \le C \sqrt{\frac{b^2 \sigma^2 t}{n}}
    ~~ \mbox{w.p.~} \ge 1 - e^{-t}
    ~~~ \mbox{and} ~~~
    \E[|\what{\theta}_n - \theta|^2]
    \le C \frac{b^2 \sigma^2}{n}
    + C e^{-c n / b^2}.
  \end{equation*}
\end{lemma}

To achieve an upper bound matching
Corollary~\ref{corollary:gaussian}, consider
the following non-interactive estimator, which provides
$\diffp$ of differential privacy. We consider the cases
$\diffp \le 1$ and $\diffp \ge 1$ separately.
\begin{enumerate}[i.]
\item In the case that $\diffp \ge 1$, choose
  $\floor{\diffp} \wedge d$ coordinates $j \in [d]$ uniformly at random. On
  each chosen coordinate $j$, release $Z_{i,j}$ via
  mechanism~\eqref{eqn:sign-flip-private-gauss} using privacy level
  $\diffp_0 = 1$, and use the estimator~\eqref{eqn:gaussian-estimator}
  applied to each coordinate;
  this mechanism is $\diffp$-differentially private, each
  coordinate (when sampled) takes values $|Z_{i,j}| \le \frac{e + 1}{e-1}$,
  and so the resulting vector $\what{\theta}_n \in \R^d$ satisfies
  \begin{equation*}
    \E[\ltwos{\what{\theta}_n - \theta}^2]
    \le \frac{C d \sigma^2}{n ((\floor{\diffp} \wedge d) / d)}
    \le C \min\left\{\frac{d^2}{n \diffp}, \frac{d}{n}\right\}.
  \end{equation*}
\item When $\diffp < 1$, we use the $\ell_\infty$-based mechanism of
  \citet{DuchiJoWa18} applied to the vector $\sgn(X_i) \in \{-1, 1\}^d$,
  which then releases a vector $Z_i \in C \sqrt{d / \diffp^2} \cdot \{-1,
  1\}^d$ for a numerical constant $C$ chosen to guarantee $\E[Z \mid
    \sgn(X)] = \sgn(X)$. Thus each coordinate of $Z_i$ satisfies the
  conditions of Lemma~\ref{lemma:gaussian-achievability}, and applying the
  inversion~\eqref{eqn:gaussian-estimator} to each coordinate independently
  yields $\E[\ltwos{\what{\theta}_n - \theta}^2] \le \frac{C d^2}{n
    \diffp^2}$. In this setting, the value $\diffpkl \le 2\diffp^2$ by
  inequality~\eqref{eqn:diffp-to-renyi}.
\end{enumerate}
  
\subsection{Sparse Estimation}
\label{sec:sparse-estimation}

We now turn to settings in which the coordinates exhibit dependence,
assuming individuals have $\diffp \le 1$-differential privacy
to make the discussion concrete. Consider the sparse Gaussian mean problem,
$X_i \simiid \normal(\theta, I_d)$ for $\lzero{\theta} = k$.  For
simplicity, let us consider that $k = 1$ and is known;
Corollary~\ref{corollary:gaussian} gives the minimax lower bound $d / (n
\diffp^2)$ under $\diffp$-differential privacy, which
\citet[Sec.~4.2.2]{DuchiJoWa18} achieve to within a logarithmic factor; the
non-private minimax risk~\citep{Johnstone13} is the exponentially smaller
$\frac{\log d}{n}$.
In the case of a (very) large sample size $n$, however, we observe a
different phenomenon: the non-private and private rates coincide.

Let us assume that $n \gg d$, and that $n
\to \infty$ as $d$ remains fixed. Let the sample be of
size $2n$, which we split. On the first
half, we further split the sample into $d$ bins of size $n/d$; for each of
these $d$ bins, we construct a 1-dimensional estimator of the mean of
coordinate $j$ via~\eqref{eqn:gaussian-estimator}, which gives us $d$
preliminary estimates $\what{\theta}_1^{\rm pre}, \ldots,
\what{\theta}_d^{\rm pre}$, each of which is $\diffp$-locally differentially
private. Lemma~\ref{lemma:gaussian-achievability}
shows that we can identify the non-zero coordinate of $\theta$ by $\what{j}
\defeq \argmax_j |\what{\theta}_j^{\rm pre}|$ with exponentially high
probability. Then, on the second half of the sample, we apply the
private estimator~\eqref{eqn:gaussian-estimator} to estimate the mean of
coordinate $\what{j}$. In combination, this yields an estimator
$\what{\theta}_{2n}$ that achieves
$\E[\ltwos{\what{\theta}_{2n} - \theta}^2] \le C / (n \diffp^2)$ for
large $n$, while the non-private analogue in this case has risk
$1/n$.

We have moved from an \emph{exponential} gap in the
dimension to one that scales only as $1 / \diffp^2$, as soon as $n$ is large
enough. This example is certainly stylized and relies on a particular
flavor of asymptotics ($n \to \infty$); we believe
this transformation from ``independent'' structure, with
risk scaling as $d/n$, to an identified structure with risk scaling as
$1/n$, merits more investigation.

\subsection{Correlated Data}
\label{sec:correlated-data}

We consider an additional stylized example of correlation.
Let $b \in \{\pm 1 \}^d$ be a \emph{known}
bit vector and assume the data $X_i = b \cdot B_i$ where $B_i \in \{\pm
1\}$, $P(B_i = 1) = p$ for an unknown $p$.  Without privacy,
$\what{p} = \frac{1 + \wb{B}_n}{2}$ achieves minimax
optimal $\ell_2^2$
risk $\frac{d}{n}$; the error is $d$ times that for the
one-dimensional quantity. In the private case, as $b \in \{\pm
1\}^d$ is known, the private channel for user $i$ may privatize only the bit
$B_i$ using randomized response, setting $Z_i$ as in
Eq.~\eqref{eqn:sign-flip-private-gauss}. Using
the private estimate $\what{p}_\diffp = \tfrac{1+\wb{Z}_n}{2}$ yields
$\E[(\what{p}_\diffp - p)^2] \le C / (n \min\{\diffp^2, 1\})$,
so
$\what{\theta}_n = b (2 \what{p}_\diffp - 1)$ has mean square
error
\begin{equation*}
  \E\left[\ltwos{\what{\theta}_n - b \cdot (2p-1)}^2\right]
  \le C d \cdot
  \E[(\what{p}_\diffp - p)^2]
  \le C \frac{d}{n \min\{\diffp^2, 1\}}.
\end{equation*}
In contrast to the case with independent coordinates
in Corollary~\ref{corollary:bernoulli}, here
the locally private estimator achieves
(to within a factor of $\diffp^{-2}$) the same risk as the
non-private estimator. This example is again special, but it suggests 
that leveraging correlation structures may close some of the
substantial gaps between private and non-private estimation
that prevent wider adoption of private estimators.

% !TEX root = communication-lower-bounds.tex

% -*- Mode: latex -*- %

\newcommand{\altpackrv}{U}
\newcommand{\altpackval}{u}

\newcommand{\separation}{\mathsf{sep}}

\section{Lower bounds via information complexity}
\label{sec:lower-bounds}

We turn to stating and proving our main minimax lower bounds,
which build out of work by \citet{ZhangDuJoWa13_nips},
\citet{GargMaNg14}, and \citet{BravermanGaMaNgWo16}
on communication limits in estimation.

%% We now turn to stating and proving our main minimax lower bounds.  The
%% analysis follows from work, beginning with \citet{ZhangDuJoWa13_nips} and
%% continuing through \citet{GargMaNg14} and
%% \citet{BravermanGaMaNgWo16}, on communication-based lower
%% bounds for estimation problems.

%These
%build out of a combination of information complexity, building out of work
%of \citet{ZhangDuJoWa13_nips} and following work by \citet{GargMaNg14} and
%\citet{BravermanGaMaNgWo16} on communication lower bounds for statistical
%estimation problems.

We begin with an extension of Assouad's method~\citep{Assouad83,Yu97}, which
transforms a $d$-dimensional estimation problem into one of testing
$d$ binary hypotheses, to information-limited settings. We consider a
family of distributions $\{P_\packval\}_{\packval \in \packset}$ indexed by
the hypercube $\packset = \{-1, 1\}^d$, where nature chooses $\packrv \in
\packset$ uniformly at random.  Conditional on $\packrv = \packval$, we
draw $\{X_i\}_{i=1}^n \simiid P_\packval$, from which we obtain the
observed (privatized) $\bZ$.  Letting $\theta_\packval =
\theta(P_\packval)$, we follow \citet{DuchiJoWa18} and say that $\packset$
induces a $\delta$-Hamming separation if there exists $\mathsf{\packval} :
\Theta \to \{-1, 1\}^d$ such that
\begin{equation}
  \label{eqn:delta-separation}
  L(\theta, \theta_\packval)
  \ge \delta \sum_{j = 1}^d \indic{\mathsf{\packval}_j(\theta)
    \neq \packval_j}.
\end{equation}
%% Let us provide two examples satisfying the separation
%% condition~\eqref{eqn:delta-separation}.

\begin{example}[Location families]
  \label{example:mean-estimation}
  Let $\mc{P}$ be a family of distributions, each specified
  by a mean $\theta(P)$, and for each
  $\packval \in \{-1, 1\}^d$ set $\theta(P_\packval) = \delta \cdot \packval$
  for some $\delta > 0$.
  Then for any symmetric $\ell : \R \to \R_+$ and
  loss of the form $L(\theta, \theta')
  = \sum_{j = 1}^d \ell(\theta_j - \theta'_j)$,
  we
  have $L(\theta, \theta_\packval)
  \ge \ell(\delta) \sum_{j = 1}^d \indic{\sgn(\theta_j) \neq \packval_j }$.
\end{example}

\ifcolt
\noindent
As our proof of
Corollary~\ref{corollary:logistic-lower-bound} demonstrates, similar
separations hold for convex risk minimization.
\else
Similar
separations hold for (strongly) convex risk minimization problems.

\begin{example}[Convex risk minimization]
  \label{example:convex-minimization-separation}
  Consider the problem of minimizing a convex risk functional
  $\risk_P(\theta) \defeq \E_P[\loss(\theta; X)]$, where $\loss$ is convex
  in its first argument and the expectation is over $X \sim P$.  Now, define
  $\theta(P) = \argmin_\theta \E[\loss(\theta; X)]$, and let $L(\theta,
  \theta(P)) = \risk_P(\theta) - \risk_P(\theta(P))$.  If $\risk_P$ is
  $\lambda$-strongly convex in a neighborhood of radius $r$ of $\theta(P)$,
  then a straightforward convexity argument~\cite{HiriartUrrutyLe93ab}
  yields
  \begin{equation*}
    \risk_P(\theta) - \risk_P(\theta(P))
    \ge \min\left\{\frac{\lambda}{2} \ltwo{\theta - \theta(P)}^2,
    \lambda r \ltwo{\theta - \theta(P)}\right\}.
  \end{equation*}
  Thus, if as in the previous example we can
  construct distributions $P$ such that
  $\theta(P_\packval) = \delta \cdot\packval \in \{-\delta, \delta\}^d$,
  where $\delta \le r$, then $L(\theta, \theta(P))$ induces
  a $\lambda \delta^2 / 2$-separation in Hamming metric.
\end{example}
\fi

Letting $\P_{+j}$ and $\P_{-j}$ be the marginal distributions of the
privatized $\bZ$
conditional on $\packrv_j = 1$ and $\packrv_j = -1$, respectively,
we have Assouad's method
(\citet[Lemma~1]{DuchiJoWa18} gives this form):
\begin{lemma}[Assouad's method]
  \label{lemma:assouad}
  Let the conditions of the previous paragraph hold and
  let $\packset$ induce a $\delta$-separation in Hamming metric.
  Then 
  \begin{equation*}
    \minimax_n(\theta(\mc{P}), L, \channel)
    \ge \delta \sum_{j = 1}^d \inf_{\what{\packrv}}
    \P(\what{\packrv}_j(\bZ) \neq \packrv_j)
    = \frac{\delta}{2}
    \sum_{j = 1}^d \left(1 - \tvnorm{\P_{+j} - \P_{-j}}\right).
  \end{equation*}
  %% where the infimum is over all
  %% tests $\psi : \mc{Z}^{nT} \to \{0, 1\}$.
\end{lemma}
\noindent
Consequently, if we can show that the total variation distance
$\tvnorms{\P_{+j} - \P_{-j}}$ is small while the
$\delta$-separation~\eqref{eqn:delta-separation} is large for our family, we
have shown a strong lower bound.

\subsection{Strong data processing and information contraction}

To prove lower bounds via Lemma~\ref{lemma:assouad}, we build off of ideas
that originate from \citet{ZhangDuJoWa13_nips}, which
\citet{BravermanGaMaNgWo16} develop elegantly.
\citeauthor{BravermanGaMaNgWo16} show how \emph{strong data processing}
inequalities, which quantify the information loss in classical
information processing inequalities~\citep{CoverTh06}, extend from one
observation to multiple observations.  They use this to prove lower bounds
on the information complexity of distributed estimators, and
we show how their results imply strong lower bounds on
private estimation.  We first provide a definition.

\begin{definition}
  \label{definition:strong-data-processing}
  Let $\altpackrv \to X \to Z$ be a Markov chain, where $\altpackrv$ takes
  values $\{-1, 1\}$, and conditional on $\altpackrv = \altpackval$ we
  draw $X \sim P_\altpackval$, then draw $Z$ conditional on $X$. The
  \emph{strong data processing constant $\sdpi(P_{-1}, P_1)$} is the
  smallest $\beta \le 1$ such that
  for all distributions $X \to Z$,
  \begin{equation*}
    I(\altpackrv; Z) \le \beta I(X; Z).
  \end{equation*}
\end{definition}
\noindent
Many distributions satisfy strong data
processing inequalities; Gaussians do~\citep{BravermanGaMaNgWo16},
as do distributions with bounded likelihood ratio
$dP_1 / dP_{-1}$ (see Lemma~\ref{lemma:sdp-for-bounded-likelihood}
in Appendix~\ref{sec:proofs-corollaries}).

We consider families of distributions where the
coordinates of $X$ are independent, dovetailing with Assouad's method.
For $\packval \in \{-1, 1\}^d$, conditional
on $\packrv = \packval$ we assume that
\begin{equation}
  \label{eqn:product-generation}
  X \sim P_\packval
  = P_{\packval_1} \otimes P_{\packval_2} \otimes \cdots \otimes
  P_{\packval_d},
\end{equation}
a $d$-dimensional product distribution.
That is, conditional on $\packrv_j = \packval_j$,
the coordinates $X_{i,j}$ are i.i.d.\ and independent
of $\packrv_{\setminus j} = (\packrv_1, \ldots, \packrv_{j-1},
\packrv_{j+1}, \ldots, \packrv_d)$.
When we have the generation strategy~\eqref{eqn:product-generation},
we can use \citeauthor{GargMaNg14} and
\citeauthor{BravermanGaMaNgWo16}'s results
to prove the following lower bound.

\begin{theorem}
  \label{theorem:assouad-information}
  Let $\packrv \in \{-1, 1\}^d$ and consider the Markov chain $\packrv \to
  X_{\leq n} \to \bZ$, where conditional on $\packrv = \packval$ the $X_i$
  are i.i.d., follow the product
  distribution~\eqref{eqn:product-generation}, and $\bZ$
  follows the protocol of Fig.~\ref{fig:communication-scheme}.  Assume that
  for each coordinate $j$, the chain $\packrv_j \to X_{i,j}$ satisfies a
  strong data processing inequality with $\sdpi(P_{-1}, P_1) = \sdpi$,
  and $|\log \frac{dP_1}{dP_{-1}}| \le b$ for some $b < \infty$.  Then for
  any estimator $\what{\packrv}$,
  \begin{equation*}
    \sum_{j = 1}^d
    \P(\what{\packrv}_j(\bZ) \neq \packrv_j)
    \ge \frac{d}{2}
    \left(1 -
    \sqrt{\frac{7 (e^b + 1)}{d} \sdpi \cdot I(X_{\leq n}; \bZ \mid \packrv)}
    \right).
  \end{equation*}
\end{theorem}
\noindent
We defer the proof of Theorem~\ref{theorem:assouad-information} to
Appendix~\ref{sec:proof-assouad-information}.
Lemma~\ref{lemma:sdp-for-bounded-likelihood} to come shows that if $|\log
\frac{dP_1}{dP_{-1}}| \le b$, then $\beta(P_{-1}, P_1) \le 2 (e^b - 1)^2$,
often allowing easier application of the theorem.

By combining Theorem~\ref{theorem:assouad-information} with
Lemma~\ref{lemma:assouad}, we can prove strong lower bounds on minimax rates
of convergence if we can both (i) provide a strong data processing constant
for $P_{-1}$ and $P_1$ and (ii) bound the mutual information $I(X_{\leq n};
\bZ \mid \packrv)$.  We do both presently, but we note that
Theorem~\ref{theorem:assouad-information} relies strongly on the repeated
communication structure in Figure~\ref{fig:communication-scheme} (as does
Corollary~\ref{corollary:braverman}, \citeauthor{BravermanGaMaNgWo16}'s
Theorem~3.1 in the sequel). Similar techniques appear challenging in
centralized settings. Key to our applications of the theorem,
which rely
on i.i.d.\ sampling of the vector $X_{\le n}$ to provide bounds on mutual
information via privacy, is that
\citeauthor{BravermanGaMaNgWo16}'s results allow us to take the information
\emph{conditional} on $\packrv$; without this our results fail.

\subsection{Information bounds}

To apply Theorem~\ref{theorem:assouad-information}, the first step is to
develop information bounds on private communication.  We present our three
main lemmas that accomplish this, based on
Assumptions~\ref{assumption:general-epsilon-privacy},
\ref{assumption:general-epsilon-delta-privacy},
and~\ref{assumption:summed-approximate-privacy} here.
As in the development of our assumptions, we divide our information bounds
into two cases, depending on whether we work in the fully interactive
or compositional privacy
setting.

%% (Sec.~\ref{sec:fully-interactive},
%% Assumptions~\ref{assumption:general-epsilon-privacy}
%% and~\ref{assumption:general-epsilon-delta-privacy}) or compositional privacy
%% settings (Sec.~\ref{sec:compositional},
%% Assumption~\ref{assumption:summed-approximate-privacy}).

\subsubsection{Information bounds for fully interactive mechanisms}

In this section, we provide the two bounds on mutual information bounds that
give our results.
Before stating them, however, we give the corollary to
Theorem~\ref{theorem:assouad-information} that they immediately
imply.
\begin{corollary}
  \label{corollary:useful-general-one}
  Let the conditions of Theorem~\ref{theorem:assouad-information}
  hold and assume additionally that the channels
  $\channel$ satisfy
  Assumption~\ref{assumption:general-epsilon-privacy}
  or~\ref{assumption:general-epsilon-delta-privacy},
  setting $\diffpkl = \min\{9 \diffp, 75 \diffp^2\}$ in this case. Then
  \begin{equation*}
    \sum_{j = 1}^d
    \P(\what{\packrv}_j(\bZ) \neq \packrv_j)
    \ge \frac{d}{2}
    \left(1 -
    \sqrt{\frac{7 (e^b + 1)}{d} \sdpi n \diffpkl}
    \right).
  \end{equation*}
\end{corollary}
\noindent
The corollary is immediate from
Lemmas~\ref{lemma:full-interactive-information-bound}
and~\ref{lemma:eps-delta-all-information} to come.
We begin with the former, which extends
\citet[Prop.\ 7 or 4.3]{McGregorMiPiReTaVa10}
and simplifies
\citet[Prop.~3.4]{FeldmanSt18}.
\begin{lemma}
  \label{lemma:full-interactive-information-bound}
  Let the channel $\channel$ and transcript satisfy
  Assumption~\ref{assumption:general-epsilon-privacy}. Then for any
  Markov chain $\packrv \to X_{\le n} \to \bZ$, where the $X_i$ are
  independent conditional on $\packrv$, we have
  \begin{equation*}
    I(\bZ; X_{\le n} \mid \packrv) \le n \cdot \diffpkl.
  \end{equation*}
\end{lemma}
\noindent
See Section~\ref{sec:proof-full-interactive-information-bound}
for the proof.

In the more complicated $(\diffp,\delta)$-differential privacy cases, we
require more care. Because of lack of space, we must defer the argument
to Appendix~\ref{sec:proof-eps-delta-all-information},
stating only the final conclusion here.
The lynchpin of our argument is based on the development of
\citet{RogersRoSmTh16}, who develop mutual information bounds for
discrete random variables under $(\diffp,\delta)$-differential privacy.
\begin{lemma}
  \label{lemma:eps-delta-all-information}
  Let the private variables $\bZ$ satisfy
  Assumption~\ref{assumption:general-epsilon-delta-privacy}. Then
  \begin{equation*}
    I(X_{\le n}; \bZ \mid \packrv)
    \le n \min\left\{9 \diffp, 75 \diffp^2 \right\}.
  \end{equation*}
\end{lemma}

\subsubsection{Information bounds for compositional mechanisms}
\label{sec:information-compositional-bounds}

The main result of the section, which follows by combining
Theorem~\ref{theorem:assouad-information} with the lemmas to come, gives the
following corollary.
\begin{corollary}
  \label{corollary:useful-one}
  Let the conditions of Theorem~\ref{theorem:assouad-information}
  hold and assume additionally that the channels $\channel$ satisfy
  Assumption~\ref{assumption:summed-approximate-privacy}. Then
  \begin{equation*}
    \sum_{j = 1}^d
    \P(\what{\packrv}_j(\bZ) \neq \packrv_j)
    \ge \frac{d}{2}
    \left(1 -
    \sqrt{\frac{7 (e^b + 1)}{d} \sdpi n \diffpkl}
    - \totaldelta
    \right).
  \end{equation*}
\end{corollary}

The corollary follows from
Lemma~\ref{lemma:full-interactive-information-bound} once we subtract
$\totaldelta$ and use the following approximation guarantee, which shows
that $(\diffp,\delta)$ channels are nearly differentially private.

\begin{lemma}
  \label{lemma:tv-eps-delta}
  Let Assumption~\ref{assumption:summed-approximate-privacy} hold
  on the channel $\channel$.
  Let $\P_{-1}$ and $\P_1$ be the marginal distributions of $\bZ$ under the
  communication model of Fig.~\ref{fig:communication-scheme} with
  channel $\channel$ and base distributions $P_{-1}$ and $P_1$ on
  $X_{\le n}$, so that $\P_v(S) = \int \channel(S \mid x_{\le n})
  dP_v(x_{\le n})$.
  For each $i, t$ there exist channels $\wb{\channel}(Z_i^{(t)}
  \in \cdot \mid x_i, \intoit{z})$ from $X_i$ to $Z_i^{(t)}$, conditional on
  $\intoit{z}$, where each channel is
  $\diffp_{i,t}(\intoit{z})$-differentially private. The induced
  marginal distributions $\wb{\P}_{-1,1}$ under the channels
  $\wb{\channel}$ satisfy
  \begin{equation*}
    \tvnorm{\P_{-1} - \P_1} \le \tvnorm{\wb{\P}_{-1} - \wb{\P}_{1}}
    + \totaldelta.
  \end{equation*}
\end{lemma}
\noindent
The most challenging part of Lemma~\ref{lemma:tv-eps-delta}
is to establish the existence of regular conditional probabilities
$\wb{\channel}$ (i.e., veryifying measurability) that are close to
$\channel$; we do so in Appendix~\ref{sec:proof-tv-eps-delta}.

\ifcolt
\else
% !TEX root = communication-lower-bounds.tex

\section{Proofs of Corollaries}
\label{sec:proofs-corollaries}

Before proving the corollaries from Section~\ref{sect:minimax}, we present
one lemma that will be useful throughout.  It is similar to, but simpler
than, a result of \citet[Lemma 8]{ZhangDuJoWa13_nips}.
\begin{lemma}
  \label{lemma:sdp-for-bounded-likelihood}
  Let $V \to X \to Z$, where
  $X \sim P_v$ conditional on $V = v$. If
  $|\log \frac{dP_v}{dP_{v'}}| \le \alpha$ for all $v, v'$, then
  \begin{equation*}
    I(V; Z)
    \le 4 (e^\alpha - 1)^2
    \E_Z[\tvnorm{P_X(\cdot \mid Z) - P_X}^2]
    \le 2 (e^\alpha - 1)^2 I(X; Z).
  \end{equation*}
\end{lemma}
\begin{proof}
  By approximation, there is no loss of generality to assume that
  each random variable is discrete~\citep{Gray90}, so that
  our variables may have probability mass functions, which we
  denote by $p$.
  We first claim that
  \begin{equation}
    |p(v \mid z) - p(v)|
    \le 2 (e^{\alpha} - 1)
    p(v)
    \tvnorm{P_X(\cdot \mid z) - P_X(\cdot)}.
    \label{eqn:turn-off-economist}
  \end{equation}
  Indeed,
  we have that $p(v \mid x) = p(x \mid v) p(v) / p(x)
  \in [e^{-\alpha}, e^\alpha] p(v)$ by assumption on $dP_v / dP_{v'}$.
  Thus, the Markov structure $V \to X \to Z$ implies
  \begin{align*}
    |p(v \mid z) - p(v)| & =
    \left|\sum_{x} p(v \mid x) p(x \mid z)
    - p(v \mid x) p(x)\right| \\
    & = \left| \sum_x
    (p(v \mid x) - p(v))(p(x \mid z) - p(x))\right| \\
    & \le |e^\alpha - 1| p(v) \sum_x |p(x \mid z) - p(x)|
    = 2 (e^\alpha - 1) p(v) \tvnorm{P_X(\cdot \mid z) - P_X}.
  \end{align*}
  Then using the definition of mutual information and that
  $\chi^2$-divergence upper bounds the KL-divergence~\cite[Lemma
    2.7]{Tsybakov09},
  \begin{align*}
    I(V; Z) & =
    \E_Z[\dkl{P_V(\cdot \mid Z)}{P_V}] \\
    & \le \E_Z\left[\sum_v
      \left(\frac{p(v \mid Z) - p(v)}{p(v)}\right)^2 p(v)\right]
    \le 4 (e^\alpha - 1)^2
    \E_Z\left[\sum_v p(v) \tvnorm{P_X(\cdot \mid Z) - P_X}^2\right],
  \end{align*}
  where the second inequality used
  inequality~\eqref{eqn:turn-off-economist}.
  By Pinsker's inequality,
  we have the bound $\tvnorm{P_X(\cdot \mid Z) - P_X}^2 \le
  \half \dkl{P_X(\cdot \mid Z)}{P_X}$, and
  using that $I(Z; X) = \E_Z[\dkl{P_X(\cdot \mid Z)}{P_X}]$
  gives the lemma.
\end{proof}

\subsection{Proof of Corollary~\ref{corollary:bernoulli}}
\label{sec:proof-bernoulli}

By Corollaries~\ref{corollary:useful-general-one}
and~\ref{corollary:useful-one}, it will be sufficient to provide a good
enough strong data processing inequality for Bernoulli random variables. We
give the proof under Assumption~\ref{assumption:summed-approximate-privacy}
(which relies on Corollary~\ref{corollary:useful-one}), as the other cases are
completely similar.  Let $P_{-1} = \bernoulli(\half)$ and, for some $\delta
< 1$, let $P_1 = \bernoulli(\frac{1 + \delta}{2})$.  Then $|\log dP_1 /
dP_{-1}| \le -\log(1 - \delta)$, and consequently, for $\packrv$ uniform on
$\{-1, 1\}$, we obtain
\begin{equation*}
  I(\packrv; Z)
  \le 2 \left(\frac{1}{1 - \delta} - 1\right)^2 I(X; Z)
  = \frac{2 \delta^2}{1 - 2 \delta + \delta^2} I(X; Z).
\end{equation*}
In particular, we have $\sdpi(P_{-1}, P_1) \le \frac{2 \delta^2}{(1 -
  \delta)^2}$, and in the notation of
Theorem~\ref{theorem:assouad-information}, we have $b = -\log\left(1 -
\delta\right)$ as well. Thus, for any $\delta < 1$, we have
\begin{equation*}
  \sum_{j = 1}^d \P(\what{V}_j(Z) \neq V_j)
  \ge \frac{d}{2}
  \left(1 - \sqrt{\frac{7 (2 - \delta)}{(1 - \delta)}
    \frac{2 \delta^2}{(1 - \delta)^2} \frac{n \diffpkl}{d}}
  - \totaldelta \right).
\end{equation*}
Taking $\delta^2 = c \min\{1, d / (n \diffpkl)\}$,
using that $\totaldelta \le \half$, and noting that the
separation is at least $\delta/2$ in Assouad's
Lemma~\ref{lemma:assouad} gives the corollary.

\subsection{Proof of Corollary~\ref{corollary:logistic-lower-bound}}
\label{sec:proof-logistic-lower-bound}

\ifcolt
We give a brief example before beginning the proof to show
that similar ideas extend to other convex risk minimization problems.
\begin{example}[Convex risk minimization]
  \label{example:convex-minimization-separation}
  Consider the problem of minimizing a convex risk functional
  $\risk_P(\theta) \defeq \E_P[\loss(\theta; X)]$, where $\loss$ is convex
  in its first argument and the expectation is over $X \sim P$.  Now, define
  $\theta(P) = \argmin_\theta \E[\loss(\theta; X)]$, and let $L(\theta,
  \theta(P)) = \risk_P(\theta) - \risk_P(\theta(P))$.  If $\risk_P$ is
  $\lambda$-strongly convex in a neighborhood of radius $r$ of $\theta(P)$,
  then a straightforward convexity argument~\citep{HiriartUrrutyLe93ab}
  yields
  \begin{equation*}
    \risk_P(\theta) - \risk_P(\theta(P))
    \ge \min\left\{\frac{\lambda}{2} \ltwo{\theta - \theta(P)}^2,
    \lambda r \ltwo{\theta - \theta(P)}\right\}.
  \end{equation*}
  Thus, if as in the previous example we can
  construct distributions $P$ such that
  $\theta(P_\packval) = \delta \cdot\packval \in \{-\delta, \delta\}^d$,
  where $\delta \le r$, then $L(\theta, \theta(P))$ induces
  a $\lambda \delta^2 / 2$-separation in Hamming metric.
\end{example}
\fi

Our proof proceeds in two steps. First, we argue that the gap in the
logistic risk is lower bounded by a quadratic
(cf.\ Example~\ref{example:convex-minimization-separation}); we then argue
that this quadratic lower bound can be reduced to estimation in a model with
independent Bernoulli coordinates. To avoid somewhat tedious constants, we
perform the analysis in an asymptotic sense.

We first describe the precise problem setting. Let $\delta > 0$, to be
chosen later, and let $\packval \in \packset \defeq \{\pm 1\}^d$ as is
standard for our applications of Assouad's method, and for each $\packval
\in \packset$ let $\theta^{\packval} = \delta \packval$.  Now, for any
$\theta \in \{\pm \delta\}^d$, consider the class-conditional
distributions with coordinates of $X \in \R^d$ independent and distributed
(conditional on $Y \in \{\pm 1\}$) as
\begin{equation*}
  X_j \mid Y = \begin{cases}
    Y & \mbox{w.p.}~ \frac{e^{\theta_j/2}}{e^{\theta_j/2} + e^{-\theta_j/2}}
    = \frac{e^{\theta_j X_j Y / 2}}{e^{\delta/2} + e^{-\delta/2}} \\
    -Y & \mbox{w.p.}~ \frac{e^{-\theta_j/2}}{e^{\theta_j/2} + e^{-\theta_j/2}}
    = \frac{e^{-\theta_j X_j Y / 2}}{e^{\delta/2} + e^{-\delta/2}}.
  \end{cases}
\end{equation*}
Let the prior probabilities $P(Y = y) = \half$ for $y \in \{\pm 1\}$
Then conditional on $X = x \in \{\pm 1\}^d$, we have
\begin{equation*}
  P(Y = y \mid X = x)
  %% = \frac{P(X = x \mid Y = y) P(Y = y)}{P(X = x)}
  = \frac{\prod_{j = 1}^d e^{\theta_j x_j y / 2}}{
    \prod_{j = 1}^d e^{\theta_j x_j y / 2}
    + \prod_{j = 1}^d e^{-\theta_j x_j y / 2}}
  = \frac{e^{\theta^T x y}}{1 + e^{\theta^T x y}},
\end{equation*}
so that $Y \mid X$ follows the logistic model.
%% where we have used that $e^{\theta_j x_j y/2}
%% + e^{-\theta_j x_j y / 2} = e^{\delta/2} + e^{-\delta/2}$ for all
%% $x, y, \theta$ in our setting.

\paragraph{Quadratic lower bounds on risk:}
Fixing $\packval$, let $\risk_{\delta\packval}(\theta) = \E_{\delta
  \packval}[ \loss(\theta; (X, Y))]$, where $\E_{\delta\packval}$
indicates expectation under the logistic model above with $\theta = \delta
\packval$; note that $\theta\opt \defeq \argmin_\theta \risk_{\delta
  \packval}(\theta) = \delta \packval$ here. We claim that for all
$\epsilon > 0$ there exists a $\gamma > 0$ such that
\begin{equation}
  \liminf_{\delta \downarrow 0}
  \inf_{\ltwo{\theta} \le \gamma} \lambda_{\min}(\nabla^2 \risk_{\delta \packval}
  (\theta)) \ge \frac{1 - \epsilon}{4}.
  \label{eqn:lower-quadratic-logistic}
\end{equation}
We return to prove inequality~\eqref{eqn:lower-quadratic-logistic}
at the end of the proof of the corollary,
noting that by Example~\ref{example:convex-minimization-separation},
it immediately implies that if $\delta > 0$ is small enough then
\begin{equation*}
  \risk_{\delta\packval}(\theta)
  - \inf_\theta \risk_{\delta\packval}(\theta)
  \ge \min\left\{\frac{1 - \epsilon}{8}
  \ltwo{\theta - \delta \packval}^2,
  \frac{1 - \epsilon}{4} \gamma \ltwo{\theta - \delta \packval}
  \right\}.
\end{equation*}
Projecting $\theta$ into
the set $[-\delta, \delta]^d$ can only decrease the right hand side
of the previous display, and thus (again for small enough $\delta > 0$ and
using that $\gamma > 0$ is fixed relative to $\delta$)
we see that
\begin{equation}
  \label{eqn:separation-in-logistic}
  \risk_{\delta\packval}(\theta) - \inf_{\theta} \risk_{\delta\packval}(\theta)
  \ge \delta^2 \frac{1 - \epsilon}{8}
  \sum_{j = 1}^d \indic{\sgn(\theta_j) \neq \packval_j}.
\end{equation}
This is exactly the separation condition~\eqref{eqn:delta-separation}
necessary for application of Assouad's method.

\paragraph{Reduction to Bernoulli estimation}
By construction, for each coordinate $j$, we have $YX_j \sim
\bernoulli(e^{\theta_j} / (1 + e^{\theta_j}))$, independent of the others.
As a consequence, we see for any estimator $\what{\packrv}$ of the
signs of the parameters of the logistic model, there exists an estimator
$\what{\packrv}^{\mathsf{bern}}$ and channel $\channel^{\mathsf{bern}}$,
which is equally private to $\channel$ (and both are independent of the
true $\theta = \delta \packval$), such that
\begin{equation}
  \label{eqn:logistic-loss-lb-bernoulli}
  \sum_{j = 1}^d \P(\what{\packrv}_j(\bZ) \neq \packval_j)
  \ge \sum_{j = 1}^d \P_{\channel^{\mathsf{bern}}}(
  \what{\packrv}_j^{\mathsf{bern}}(\bZ) \neq \packval_j),
\end{equation}
where the first expectation is taken over our logistic model with
parameters $\theta$ and the second over the distribution on $X$ with
independent $\bernoulli(e^{\theta_j} / (1 + e^{\theta_j}))$ coordinates.

We now apply an argument completely parallel to that in the proof of
Corollary~\ref{corollary:bernoulli}, again focusing on
Assumption~\ref{assumption:summed-approximate-privacy} for simplicity---the
parallel case under Assumptions~\ref{assumption:general-epsilon-privacy}
or~\ref{assumption:general-epsilon-delta-privacy} is similarly immediate
from Corollary~\ref{corollary:useful-general-one}.  Let $P_{-1} =
\bernoulli(e^{-\delta} / (1 + e^{-\delta}))$ and $P_1 = \bernoulli(e^\delta
/ (1 + e^\delta))$. Then $|\log dP_1 / dP_{-1}| \le 2 \delta$, and
Lemma~\ref{lemma:sdp-for-bounded-likelihood} implies that the strong data
processing constant $\beta(P_1,P_{-1}) \le 2 (e^{2 \delta} - 1)^2$.
Randomizing over $\packrv$ uniform in $\packset$,
Lemma~\ref{lemma:tv-eps-delta} and Theorem~\ref{theorem:assouad-information}
(coupled with Corollary~\ref{corollary:useful-one}) yield the lower bound
\begin{align*}
  \sum_{j = 1}^d \P(\what{\packrv}_j(\bZ) \neq \packrv_j)
  & \ge \frac{d}{2} \left(1
  - \sqrt{\frac{14 (e^{2 \delta} + 1)}{d}
    (e^{2 \delta} - 1)^2 I(X_{\le n}; \bZ \mid \packrv)}
  - \totaldelta\right) \\
  & \ge \frac{d}{2} \left(1
  - \sqrt{\frac{14 (e^{2 \delta} + 1)}{d}
    (e^{2 \delta} - 1)^2 n \diffpkl} - \totaldelta\right).
\end{align*}
Setting $\delta^2 = c \frac{d}{n \diffpkl}$ for small enough
constant $c > 0$, inequality~\eqref{eqn:separation-in-logistic}
coupled with inequality~\eqref{eqn:logistic-loss-lb-bernoulli}
immediately yields
\begin{equation*}
  \E\big[\risk_{\delta \packrv}(\what{\theta}_n(\bZ))
    - \inf_\theta \risk_{\delta \packrv}(\theta)\big]
  \ge
  C d \delta^2
  = C' d \frac{d}{n \diffpkl}
\end{equation*}
as desired, where $C, C' > 0$ are numerical constants.

\paragraph{Proof of inequality~\eqref{eqn:lower-quadratic-logistic}:}
As $\theta \mapsto \nabla^2 \risk_{\delta\packval}(\theta)$ is
$\mc{C}^\infty$ in $\theta$,
as is $\delta \mapsto \risk_{\delta\packval}(\theta)$ by the logistic model,
we may swap the limit infimum and infimum over $\ltwo{\theta} \le \gamma$.
Now, fix any $\theta$ with $\ltwo{\theta} \le \gamma$, where we will
choose $\gamma$ momentarily.
Then Lebesgue's dominated convergence theorem and the
continuity of the minimum eigenvalue $\lambda_{\min}$ gives
\begin{equation*}
  \liminf_{\delta \downarrow 0}
  \lambda_{\min}(\nabla^2 \E_{\delta \packval}[\loss(\theta; (X, Y))])
  =
  \lambda_{\min}(\E[p_\theta(X)(1 - p_\theta(X)) XX^T])
\end{equation*}
where $X \sim \uniform(\{\pm 1\}^d)$ and $p_\theta(X) = 1 / (1 +
e^{\theta^T X})$. As $\theta^T X$ is
$\ltwo{\theta}^2$-sub-Gaussian~\citep{Vershynin12}, meaning that
$\E[e^{\theta^T X}] \le \exp(\ltwo{\theta}^2 / 2)$, standard sub-Gaussian
concentration inequalities and that $\ltwo{\theta} \le \gamma$ imply that
with probability over at least $1 - \alpha$ over $X$, we have
$|\<\theta, X\>| \le \sqrt{2 \gamma^2 \log(2/\alpha)}$.
Setting $t = \sqrt{2 \gamma^2 \log(2/\alpha)}$,
if $\gamma > 0$ is small enough that $e^t / (1 + e^t)^2 \ge (1 - \epsilon/2)/4$
we have
\begin{equation*}
  \lambda_{\min}(\E[p_\theta(X)(1 - p_\theta(X)) XX^T])
  \ge \lambda_{\min}\left(\frac{1 - \epsilon / 2}{4} \E[XX^T]\right)
  - d \alpha
  = \frac{1 - \epsilon/2}{4} \lambda_{\min}(I_{d \times d})
  - d \alpha.
\end{equation*}
Choosing $\alpha$ and $\gamma$ small enough, we have
$\lambda_{\min}(\E[p_\theta(X) (1 - p_\theta(X))]) \ge \frac{1 - \epsilon}{4}$
as desired.

\subsection{Proof of Corollary~\ref{corollary:gaussian}}
\label{sec:proof-gaussian}

We provide a slightly different proof, beginning with the reduction
to Assouad's method. Let $\delta > 0$ to be chosen
presently. We first observe that
if $\theta \in [-\delta, \delta]^d$, then it is no loss of generality
to assume the estimator $\what{\theta} \in [-\delta, \delta]^d$, as otherwise,
we may simply project to $[-\delta, \delta]^d$. Then for any
distributions $P$ and $\wb{P}$ and any coordinate $j$, we have
\begin{equation*}
  \E_P[(\what{\theta}_j - \theta_j)^2]
  = \E_{\wb{P}}[(\what{\theta}_j - \theta_j)^2]
  + \int(\what{\theta}_j - \theta_j)^2 (dP - d\wb{P})
  \ge \E_{\wb{P}}[(\what{\theta}_j - \theta_j)^2]
  - 8 \delta^2 \tvnorm{P - \wb{P}}.
\end{equation*}
Now, let $\mc{P}_\delta$ be the collection of normal distributions with
means in $[-\delta, \delta]$. Using Lemma~\ref{lemma:tv-eps-delta}, we then
obtain that for any channel $\channel$ satisfying
Assumption~\ref{assumption:summed-approximate-privacy}, there exist
$\diffp_{i,t}$-differentially private channels $\wb{\channel}$ satisfying
$\sum_{i,t} \min\{\diffp_{i,t}, \diffp_{i,t}^2\} \le n \diffpkl$ such that
\begin{equation}
  \label{eqn:minimax-from-approx-to-pure}
  \minimax_n(\theta(\mc{P}), \ltwo{\cdot}^2,
  \channel)
  \ge \minimax_n(\theta(\mc{P}_\delta), \ltwo{\cdot}^2,
  \channel)
  \ge \minimax(\theta(\mc{P}_\delta), \ltwo{\cdot}^2,
  \wb{\channel}) - 8 d \delta^2 \totaldelta.
\end{equation}

In the lower bound~\eqref{eqn:minimax-from-approx-to-pure},
choosing
\begin{equation*}
  \delta^2 = c \min\left\{\frac{d}{\diffpkl} \frac{d \sigma^2}{n},
  1\right\}
\end{equation*}
and using Theorem 4.5 of \citet{BravermanGaMaNgWo16} (with the choice $k =
d/2$ in their result, along with the specified separation $\delta$), coupled
with Lemma~\ref{lemma:full-interactive-information-bound}, we obtain the
lower bound $c \min\{\frac{d}{\diffpkl} \frac{d \sigma^2}{n}, d\}$. The $d
\sigma^2 / n$ term is the standard minimax bound for estimation of a
Gaussian mean.

\subsection{Proof of Corollary~\ref{corollary:sparse-gaussian}}
\label{sec:proof-sparse-gaussian}

The proof is nearly identical to that of Corollary~\ref{corollary:gaussian},
except that in the lower bound~\eqref{eqn:minimax-from-approx-to-pure}, we
may replace the $8 d \delta^2 \totaldelta$ term with $16 k \delta^2
\totaldelta$, which follows by assuming w.l.o.g.\ that $\what{\theta}$ is
$k$-sparse, in which case we estimate at most $2k$ entries of $\theta$
incorrectly. Then the lower bound of $\frac{d}{\diffpkl} \frac{k
  \sigma^2}{n}$ follows by Theorem~4.5 of \citet{BravermanGaMaNgWo16},
coupled with Lemma~\ref{lemma:full-interactive-information-bound}.  The
minimum involving $k$ follows because $\ltwo{\theta - \theta'}^2 \le 4k$ for
all $\theta, \theta' \in [-1, 1]^d$ with $\lzero{\theta} \le k$. The $k
\sigma^2 \log( \frac{d}{k}) / n$ term is the standard minimax lower bound
for sparse Gaussian sequence estimation~\citep{Johnstone13}.

\fi

% !TEX root = communication-lower-bounds.tex

% -*- Mode: latex -*- %

\section{Conclusion}

By building off of the results in information-limited
statistical estimation that \citet{ZhangDuJoWa13_nips},
\citet{GargMaNg14}, and \citet{BravermanGaMaNgWo16} establish,
we have developed fundamental
limits for locally private estimation at all privacy levels and for all the
acceptable and common models of privacy.  We do not believe this paper
closes any doors, however: there is a substantial gap between the worst-case
minimax bounds and asymptotic results, highlighted by the challenges of
correlated data. Identifying structures we can leverage for more efficient
private or information-constrained estimation---an analogue of the geometric
theory available in the case of classical statistics, where Fisher
information and related ideas play an essential
role---presents a challenging
direction that, we hope, may allow more frequent practical
use of private procedures.

%% We have shown that we can tailor the results in the line of work on
%% information-based complexity of statistical estimation to show lower bounds
%% for locally private estimation.  The generality of the results in
%% \cite{ZhangDuJoWa13_nips,GargMaNg14,BravermanGaMaNgWo16} allows for
%% arbitrary interactions between machines with samples subject to information
%% constraints.  We then bound the information released from differentially
%% private protocols, and its variants, to provide minimax lower bounds in
%% estimation under privacy constraints.  Our results, as well as the previous
%% results in this line of work, crucially assumes independence across
%% coordinates in a $d$-dimensional estimation problem.  However, when
%% correlations exist, we provide an example where we can achieve must faster
%% rates of convergence.  This leaves the problem open of achieving optimal
%% rates of convergence subject to local privacy in the presence of
%% correlations across coordinates in the object we are attempting to estimate.

\paragraph{Acknowledgments}
We thank Vitaly Feldman, Aleksandar Nikolov, Aaron Roth, Adam Smith, and
Salil Vadhan for clarifying discussions and feedback on an earlier version
of this draft, which (among other things) led us to the general
Definition~\ref{definition:kl-stability} of local privacy. We also thank the
Simons Institute for hosting our visit as part of the \emph{Data Privacy:
  Foundations and Applications} semester.

\clearpage 
\ifcolt
\bibliography{bib}
\else
\setlength{\bibsep}{1.pt}
\bibliography{bib}
\bibliographystyle{abbrvnat}
\fi

% \newpage

\appendix

% !TEX root = communication-lower-bounds.tex

\section{Proof of Theorem~\ref{theorem:assouad-information}}
\label{sec:proof-assouad-information}

%% \begin{figure}
%%   \begin{center}
%%     \begin{overpic}[width=.6\columnwidth]{%,grid]{%
%%         Figures/graphical-clique-z}
%%       \put(48,8){\Large $Z$}
%%       \put(3.5,31){$X^1$}
%%       \put(25,31){$X^2$}
%%       \put(47,31){$X^3$}
%%       \put(48.4,47){$\packrv$}
%%       \put(91.5,31){$X^n$}
%%     \end{overpic}
%%   \end{center}
%%   \caption{\label{fig:generation-scheme} Generation scheme for data.}
%% \end{figure}

Our proofs build essentially directly out of the work of \citet{GargMaNg14}
and \citet{BravermanGaMaNgWo16}.  The starting point for all of these
results is a due to \citet{BravermanGaMaNgWo16}, where we have carefully
controlled the constants.
\begin{corollary}[\citet{BravermanGaMaNgWo16},
    Theorem 3.1]
  \label{corollary:braverman}
  Consider a Markov chain $\altpackrv \to Y_{\leq n} \to Z$, where
  $\altpackrv \in \{\pm 1\}$ is uniform, and
  $Y_i \simiid P_\altpackval$ conditional on $\altpackrv = \altpackval$.
  Assume that $|\log \frac{dP_1}{dP_{-1}}| \le b$ and that
  the strong data processing inequality constant of
  $P_1, P_{-1}$ is $\sdpi(P_{-1}, P_1)$. Let $M_1$ and $M_{-1}$ denote
  the marginal distributions on $Z$ conditional on
  $\altpackrv = 1$ or $-1$, respectively. Then
  \begin{equation*}
    \dhel^2(M_{-1}, M_1)
    \le \frac{7}{2} (e^b + 1)
    \sdpi(P_{-1}, P_1) \sum_{i = 1}^n \min\{I(Y_i; Z \mid \altpackrv = -1),
    I(Y_i; Z \mid \altpackrv = 1)\}.
  \end{equation*}
\end{corollary}

\noindent
The $Y_i$ are i.i.d.\ conditional on $\altpackrv$ in the corollary,
so as an immediate consequence, we have
\begin{equation}
  \dhel^2(M_{-1}, M_1)
  \le \frac{7}{2} (e^b + 1) \sdpi(P_{0}, P_1)
  \min\{I(Y_{\leq n}; Z \mid \altpackrv = -1),
  I(Y_{\leq n}; Z \mid \altpackrv = 1)\}
  \label{eqn:multi-machine-sdpi}
\end{equation}
for any $\altpackrv \to Y_{\le n} \to Z$ when
the $Y_i$ are conditionally independent given $\altpackrv$.
To see this, note that
\begin{align*}
  I(Y_{\leq n}; Z \mid \altpackrv = \altpackval)
  & = \sum_{i = 1}^n H(Y_i \mid Y_{<i}, \altpackrv = \altpackval)
  - H(Y_i \mid Y_{<i}, Z, \altpackrv = \altpackval) \\
  & \ge \sum_{i = 1}^n H(Y_i \mid \altpackrv = \altpackval)
  - H(Y_i \mid Z, \altpackrv = \altpackval)
  = \sum_{i = 1}^n I(Y_i; Z \mid \altpackrv = \altpackval),
\end{align*}
where we use $H(Y_i \mid Y_{<i}, \altpackrv = \altpackval) =
H(Y_i \mid \altpackrv = \altpackval)$ and that conditioning reduces entropy.

The key in Theorem~\ref{theorem:assouad-information}, which uses the chain
$\packrv \to X_{\le n} \to \bZ$, is (as in the case of \citet{GargMaNg14}
and \citet{BravermanGaMaNgWo16}) that each individual $i$ draws coordinate
$j$ in $X_{i,j}$ conditional on only coordinate $\packrv_j$ of $\packrv \in
\{-1, 1\}^d$, that is, \emph{independently} of $\packrv_{\setminus j}$.
Now, let $X_{\le n,j} = (X_{i,j})_{i=1}^n$ be the $j$th coordinate of the
data, and let $X_{\le n, \setminus j}$ denote the remaining $d - 1$
coordinates across all $i = 1, \ldots, n$.  By construction of our product
sampling distribution~\eqref{eqn:product-generation},
we thus have
Markov structure
\begin{equation*}
  \packrv_j \to X_{\leq n,j} \to \bZ \leftarrow X_{\le n,\setminus j}
  \leftarrow \packrv_{\setminus j},
\end{equation*}
in turn implying (by marginalizing over $X_{\le n,\setminus j}$ and
$\packrv_{\setminus j}$) the Markov structure
\begin{equation}
  \label{eqn:simple-markov-for-info}
  \packrv_j \to X_{\leq n,j} \to \bZ.
\end{equation}

Now, define $M_{\pm j}$ to be the marginal distributions
over the total communicated private variables $\bZ$
conditional on $\packrv_j = \pm 1$.
Then Le Cam's inequalities and Cauchy-Schwarz imply that
\begin{align}
  \label{eqn:get-summed-probs-to-hellinger}
  2\sum_{j = 1}^d \P(\what{\packrv}_j(\bZ) \neq \packrv_j)
  \ge \sum_{j = 1}^d (1 - \tvnorm{M_{-j} - M_{+j}})
  & \ge \sum_{j = 1}^d (1 - \sqrt{2} \dhel(M_{-j}, M_{+j})) \\
  & \ge d\left(1 - \sqrt{\frac{2}{d} \sum_{j = 1}^d \dhel^2(M_{-j}, M_{+j})}
  \right).
  \nonumber
\end{align}

It remains to bound the summed Hellinger distances.
By inequality~\eqref{eqn:multi-machine-sdpi} and the particular
Markov structure~\eqref{eqn:simple-markov-for-info}, we have
\begin{equation}
  \label{eqn:application-multi-machine-sdpi}
  \dhel^2(M_{-j}, M_{+j})
  \le \frac{7}{2} (e^b + 1) \sdpi(P_{0}, P_1)
  I(X_{\leq n,j}; \bZ \mid \packrv_j).
\end{equation}
Using the fact that conditioning reduces entropy and that conditional
on $\packrv_j$, the values $X_{\le n, j}$ are i.i.d.\ and independent
of $\packrv_{\setminus j}$, we have
\begin{align*}
  I(X_{\le n, j}; \bZ \mid \packrv_j)
  & = H(X_{\le n, j} \mid \packrv_j) - H(X_{\le n, j} \mid \packrv_j,
  \bZ) \\
  & \le H(X_{\le n, j} \mid \packrv_j, \packrv_{\setminus j})
  - H(X_{\le n, j} \mid \packrv_j, \packrv_{\setminus j}, \bZ) \\
  & =
  I(X_{\le n, j}; \bZ \mid \packrv).
\end{align*}
The following lemma relates the individual informations
to the global information $I(X_{\leq n}; \bZ \mid \packrv)$.

\begin{lemma}
  \label{lemma:coordinate-complexity-to-global}
  Let $\packrv, X_{\leq n}, \bZ$ be as in
  Theorem~\ref{theorem:assouad-information}.  Then
  \begin{equation*}
    \sum_{j = 1}^d I(X_{\le n, j}; \bZ \mid \packrv)
    \le I(X_{\leq n}; \bZ \mid \packrv).
  \end{equation*}
\end{lemma}
\begin{proof}
  We have
  \begin{align*}
    \sum_{j = 1}^d I(X_{\le n, j}; \bZ \mid \packrv)
    & = \sum_{j = 1}^d \left[H(X_{\le n, j} \mid \packrv)
      - H(X_{\le n, j} \mid \bZ, \packrv)\right] \\
    & \stackrel{(i)}{=} H(X_{\leq n} \mid \packrv)
    - \sum_{j = 1}^d H(X_{\le n, j} \mid \bZ, \packrv) \\
    & \stackrel{(ii)}{\le} H(X_{\leq n} \mid \packrv)
    - \sum_{j = 1}^d H(X_{\le n, j} \mid X_{\le n, < j}, \bZ, \packrv)
    = I(X_{\leq n}; \bZ \mid \packrv),
  \end{align*}
  where the equality~$(i)$ follows because conditional on $\packrv$, the
  coordinates $X_{\le n, j}$ are independent, and inequality~$(ii)$ because
  conditioning reduces entropy.
\end{proof}

Substituting the bound of Lemma~\ref{lemma:coordinate-complexity-to-global}
via the consequence~\eqref{eqn:application-multi-machine-sdpi} of the strong
data processing inequality~\eqref{eqn:multi-machine-sdpi} into
inequality~\eqref{eqn:get-summed-probs-to-hellinger}, we have
\begin{equation*}
  2 \sum_{j = 1}^d \P(\what{\packrv}_j(\bZ)
  \neq \packrv_j)
  \ge d \left(1 - \sqrt{7 (e^b + 1) \sdpi
    I(X_{\leq n}; \bZ \mid \packrv) / d}\right).
\end{equation*}
This is the desired result.

%% Thus
%% \begin{equation*}
%%   I(X_j^{(1:n}); \simdata_j \mid \packrv_j)
%%   \le I(X_{\le n, j}; \simdata_j \mid \packrv_j, \simify{\packrv}_{\setminus j})
%%   = I(X_{\le n, j}; \bZ \mid V)
%% \end{equation*}

% -*- Mode: latex -*- %

\section{Proofs of mutual information bounds}

\subsection{Proof of Lemma~\ref{lemma:full-interactive-information-bound}}
\label{sec:proof-full-interactive-information-bound}

We have
\begin{align}
  I(\bZ; X_{\le n} \mid \packrv)
  & = \sum_{i = 1}^n I(\bZ; X_i \mid X_{< i}, \packrv) \nonumber \\
  & = \sum_{i = 1}^n
  \E\left[\E\left[\dkl{\channel(\bZ \in \cdot \mid X_{\le i}, \packrv)}{
        \channel(\bZ \in \cdot \mid X_{< i}, \packrv)}
      \mid \packrv\right]
    \right]
  \label{eqn:information-to-kl-easy}
\end{align}
where the first equality is the chain rule and the second equality
uses the equivalence of mutual information and expected KL-divergence,
where $\channel(\bZ \in \cdot \mid X_{\le i}, \packrv)$ denotes the
conditional distribution of the full set of private variables $\bZ$ given
$X_{\le i}$.
Now we note that
\begin{equation*}
  \channel(\bZ \in \cdot \mid x_{\le i}, \packval)
  = \int \channel(\bZ \in \cdot \mid x_{\le n}) dP_\packval(x_{i+1})
  \cdots dP_\packval(x_n)
\end{equation*}
because the $X_i$ are independent conditional on $\packrv = \packval$,
and similarly for $\channel(\bZ \in \cdot \mid x_{< i}, \packval)$.
The joint convexity of the KL-divergence then implies
\begin{align*}
  \lefteqn{\dkl{\channel(\bZ \in \cdot \mid x_{\le i}, \packval)}{
      \channel(\bZ \in \cdot \mid x_{< i}, \packval)}} \\
  & \le \int_{\mc{X}^{n - i}} \int_{\mc{X}}
  \underbrace{\dkl{\channel(\bZ \in \cdot \mid x_{\le n})}{
      \channel(\bZ \in \cdot \mid x_{< i}, x_i', x_{> i})}}_{\eqdef
    \diffp_i(x_{\le n}, x_i')}
  dP_\packval(x_{i+1}) \cdots dP_\packval(x_n)
  dP_\packval(x_i')
\end{align*}
where we let $\diffp_i$ be as above.
Assumption~\ref{assumption:general-epsilon-privacy} gives
that $\sum_{i = 1}^n \diffp_i(x_{\le n}, x_i') \le n \diffpkl$,
and substituting in the chain rule~\eqref{eqn:information-to-kl-easy}
gives the result.
%% Now we note by
%% the quasi-convexity of ratios of linear functions,
%% the Markovian structure $\packrv \to X_{\le n} \to \bZ$, and
%% that conditional on $\packrv$, $X_{> i}$ is independent of
%% $X_{\le i}$, that
%% \begin{equation}
%%   \label{eqn:all-channels-are-private}
%%   \frac{\channel(\bZ \in S \mid x_{< i}, x_i, \packval)}{
%%     \channel(\bZ \in S \mid x_{< i}, x_i', \packval)}
%%   = 
%%   \frac{\int_{\mc{X}^{n-i}}
%%     \channel(\bZ \in S \mid x_{< i}, x_i, x_{> i}) dP_\packval(x_{> i})}{
%%     \int_{\mc{X}^{n-i}}
%%     \channel(\bZ \in S \mid x_{< i}, x_i', x_{> i}) dP_\packval(x_{> i})}
%%   \le e^\diffp,
%% \end{equation}
%% which is valid for all $S$ and $x_i, x_i' \in \mc{X}$.  Now we use the
%% bound~\cite[Lemma~III.2]{DworkRoVa10} that for any pair of
%% distributions $P$ and $Q$ satisfying $|\log \frac{dP}{dQ}| \le \diffp$, we
%% have $\dkl{P}{Q} + \dkl{Q}{P} \le \diffp(e^\diffp - 1)$ and $\dkl{P}{Q}
%% \le \diffp$. Thus, using inequality~\eqref{eqn:all-channels-are-private}
%% in the tensorized information~\eqref{eqn:information-to-kl-easy},
%% we have
%% $I(\bZ; X_{\le n} \mid \packrv) \le n (\diffp \wedge \diffp (e^\diffp - 1))$
%% as desired.

% -*- Mode: latex -*- %

\subsection{Proof of Lemma~\ref{lemma:eps-delta-all-information}}
\label{sec:proof-eps-delta-all-information}

The result actually follows from two more sophisticated lemmas, which we
state here and proof subsequently (see
Section~\ref{sec:proof-eps-delta-information}).
\begin{lemma}
  \label{lemma:eps-delta-information}
  Let $X_i$ be i.i.d.\ and $Z$ be $(\diffp, \delta)$-differentially
  private for $X_{\le n}$, where each $X_i$ takes values
  in the finite set $\mc{X}$.
  Let $\eta > 0$ and define
  $p_\eta = 2 (\frac{\delta}{\eta}
  + \eta \frac{e^{3\diffp}}{e^{3 \diffp} - 1}
  + \frac{\delta e^\diffp}{e^\diffp - 1})$ and the
  binary entropy $h_2(p) = -p\log p - (1 - p) \log(1 - p)$. If
  $p_\eta \le 1$, then
  \begin{equation*}
    I(X_{\le n}; Z) \le
    n \cdot \left[6 \diffp + p_\eta \log|\mc{X}| + h_2(p_\eta)\right]
  \end{equation*}
  Additionally, if $\eta > 0$ is small enough that
  $\eta (2 e^{6 \diffp} / (e^{3 \diffp} - 1) + 1) \le \half$, then
  \begin{equation*}
    I(X_{\le n}; Z)
    \le n \cdot
    \left(6 \diffp (e^{6 \diffp} - 1)
    +
    3 \eta \left[e^{3 \diffp}
      + 3 \eta \frac{e^{12 \diffp}}{(e^{3 \diffp} - 1)^2}\right]
    + p_\eta \log |\mc{X}| + h_2(p_\eta)\right).
  \end{equation*}
\end{lemma}

Extending this lemma for particular $\delta$ allows us to
provide more intepretable results.
\begin{lemma}
  \label{lemma:simplify-eps-delta-information}
  In addition to the conditions of Lemma~\ref{lemma:eps-delta-information},
  assume that
  $16 \sqrt{\delta \max\{\diffp^{-1}, 1\}} \le 1$,
  $\delta \max\{\diffp^{-1}, 1\} \log \frac{1}{\delta
    \max\{\diffp^{-1}, 1\}} \le \diffp^2$, and
  $\delta \max\{\diffp^{-1}, 1\} \log^2 |\mc{X}| \le \diffp^2$.
  Then
  \begin{equation*}
    I(X_{\le n}; Z) \le 9 n \diffp.
  \end{equation*}
  If $\diffp \le 1/6$ and we additionally have
  $\delta \le \frac{\diffp^5}{64 \log^2 |\mc{X}|}$
  and $\delta \log^2 \frac{\diffp}{\delta}
  \le \diffp^5 / 16$, then
  \begin{equation*}
    I(X_{\le n}; Z) \le 75 n \diffp^2.
  \end{equation*}
\end{lemma}
\noindent
The proof is mostly algebraic manipulations; see
Section~\ref{sec:proof-simplify-eps-delta-information}.

By recalling that in our packing of the hypercube,
the Markov chain $\packrv \to X_{\le n} \to \bZ$ guarantees that
the $X_i$ are i.i.d.\ conditional on $\packrv$,
Lemma~\ref{lemma:simplify-eps-delta-information} implies
Lemma~\ref{lemma:eps-delta-all-information} immediately.

\subsection{Proof of Lemma~\ref{lemma:eps-delta-information}}
\label{sec:proof-eps-delta-information}

In this section, we provide the proof of
Lemma~\ref{lemma:eps-delta-information}.  We require a number of different
claims.  First, we assume w.l.o.g.\ that all random variables of interest
are discrete and finitely supported (as we note earlier, the mutual
information $I(X; Y)$ is arbitrarily approximated by finite partitions of
the ranges of $X$ and $Y$~\cite{Gray90}). We make a few definitions and
give examples.

\begin{definition}
  \label{definition:indistinguishable}
  Let $X, Y$ be arbitrary random variables. They are
  \emph{$(\diffp,\delta)$-indistinguishable}, which we denote
  $X \indistinguishable{\diffp,\delta} Y$, if
  the set $E \defeq \{x : |\log \frac{P(X = x)}{P(Y = x)}|
  \le \delta\}$ satisfies $P(Y \not\in E) \ge 1 - \delta$ and
  $P(X \not \in E) \le \delta$.
\end{definition}

With this definition, we introduce a few notational shorthands for ease of
use later. Let $Z$ be the random variable distributed as
$\channel_Z(\cdot \mid X)$ (i.e.\ conditional on $X$), and we let
$\given{X}_{Z = z}$ be the random variable $X$ conditional on $Z = z$,
that is, the posterior on $X$ given $Z = z$.  With this, we can follow
\citet{RogersRoSmTh16} and their development of mutual information bounds
based on approximate differential privacy. The key is
to bound the sequence of \emph{privacy loss} random variables,
\begin{equation*}
  \privloss_i(x_{\le i}, z) \defeq
  \log \frac{P(X_i = x \mid Z = z, X_{< i} = x_{< i})}{P(X_i = x)},
\end{equation*}
as the mutual information between discrete variables $(X_{\le n}, Z)$
where the $X_i$ are i.i.d.\ is
\begin{align}
  I(X_{\le n}; Z) & = \sum_{i = 1}^n
  I(X_i; Z \mid X_{< i})
  = \sum_{i = 1}^n \E\left[\privloss_i(X_{\le i}, Z)\right]
  \label{eqn:information-as-privacy-loss}.
\end{align}

We begin with two of \citeauthor{RogersRoSmTh16}'s claims, which in turn
build off of \citet{KasiviswanathanSm14}.  For $\delta > 0$ and
$i \in [n]$, define the sets
\begin{align*}
  E_i(\delta) & \defeq \left\{(x_{< i}, z) \in \mc{X}^{i-1} \times
  \mc{Z} :
  X_i \indistinguishable{3\diffp, \delta}
  \given{X_i}_{Z = z, X_{< i} = x_{< i}}\right\} \\
  F_i & \defeq \left\{(x_{\le i}, z) \in \mc{X} \times \mc{Z}
  : |\privloss_i(x_{\le i}, z)| \le 6 \diffp \right\} \\
  G_i(\delta) & \defeq
  \left\{(x_{\le i}, z) \in \mc{X}^i \times \mc{Z}
  : (x_{< i}, z) \in E_i(\delta), (x_{\le i}, z) \in F_i \right\},
\end{align*}
so that $G$ is essentially the ``good'' set where the pair
$(X, Z)$ behaves as though $Z$ is $\diffp$-differentially private.

We then have
\begin{lemma}[\citet{RogersRoSmTh16},
    Claims 3.4--3.6]
  \label{lemma:channel-conditions-nicely}
  Let the channel $\channel_Z(\cdot \mid X)$ be
  $(\diffp, \delta)$-differentially private. Then
  for any $\eta > 0$ and $z \in E(\eta)$,
  \begin{subequations}
    \begin{align}
      \P\left((X_{< i}, Z) \in E_i(\eta) \mid X_{< i} = x_{< i}\right)
      & \ge 1 - \frac{2 \delta}{\eta}
      - \frac{2 \delta e^\diffp}{e^\diffp - 1}
      \label{eqn:equal-set-highprob} \\
      \P\left((X_{\le i}, Z) \in F_i \mid Z = z, X_{< i} = x_{< i}\right)
      & \ge 1 - \frac{2 \eta e^{3 \diffp}}{e^{3 \diffp} - 1}
      \label{eqn:bounded-set-highprob} \\
      \P\left((X_{\le i}, Z) \in G_i(\eta)\right)
      & \ge 1 - \frac{2 \delta}{\eta}
      - \frac{2 \delta e^\diffp}{e^\diffp - 1}
      - \frac{2 \eta e^{3 \diffp}}{e^{3\diffp} - 1}.
      \label{eqn:good-set-highprob}
    \end{align}
  \end{subequations}
\end{lemma}

With Lemma~\ref{lemma:channel-conditions-nicely}, we can bound
the mutual information between $X$ and $Z$.
We begin by decomposing the mutual information into two sums, as
for any $\eta > 0$,
\begin{align}
  \nonumber
  \lefteqn{I(X_i; Z \mid X_{< i})} \\
  & = \E[\privloss_i(X_{\le i}, Z) \indic{(X_{\le i}, Z) \in G_i(\eta)}]
  + \E[\privloss_i(X_{\le i}, Z) \indic{(X_{\le i}, Z) \not \in G_i(\eta)}].
  \label{eqn:split-info-into-two}
\end{align}
We control each of the terms in turn.
\begin{lemma}
  \label{lemma:control-ungood-part}
  Let $\eta > 0$, and define the shorthands
  $P(G_\eta^c) = P((X_{\le i}, Z) \not \in G_i(\eta))$
  and $h_2(p) = p \log \frac{1}{p} + (1 - p) \log \frac{1}{1 - p}$. Then
  \begin{equation*}
    \E[\privloss_i(X_{\le i}, Z) \indic{(X_{\le i}, Z) \not \in G_i(\eta)}]
    \le P(G_\eta^c) \log |\mc{X}|
    + h_2(P(G_\eta^c)).
  \end{equation*}
\end{lemma}
\begin{proof}
  For shorthand, let $G \equiv G_i(\eta)$.
  Let $X' = \given{X_i}_{(X_{\le i}, Z) \not \in G}$ and
  $Z' = \given{Z}_{(X_{\le i}, Z) \not \in G}$. Then
  as $\mc{X}$ is finite,
  we have
  \begin{align*}
    \log |\mc{X}|
    & \ge H(X_i' \mid X_{< i} = x_{< i}) \ge I(X_i'; Z_i' \mid X_{< i}
    = x_{< i}) \\
    %% & = \sum_{(x, z) \not \in G} \frac{P(X = x, Z = z)}{
    %%   P((X, Z) \not \in G)}
    %% \log \frac{P(X = x, Z = z)}{
    %%   P((X, Z) \not \in G) P(X = x \mid (X, Z) \not \in G)
    %%   P(Z = z \mid (X, Z) \not \in G)} \\
    & = \sum_{x, z}
    \frac{P(X_i = x, Z = z, G^c \mid x_{< i})}{P(G^c \mid x_{< i})}
    \log \frac{P(X_i = x, Z = z \mid x_{< i}) P(G^c \mid x_{< i})}{
      P(X_i = x, G^c \mid x_{< i})
      P(Z = z, G^c \mid x_{< i})} \\
    & \ge \sum_{x, z}
    \frac{P(X_i = x, Z = z, G^c \mid x_{< i})}{P(G^c \mid x_{< i})}
    \log \frac{P(X_i = x, Z = z \mid x_{< i}) P(G^c \mid x_{< i})}{
      P(X_i = x \mid x_{< i}) P(Z = z \mid x_{< i})} \\
    & = \sum_{x, z}
    \frac{P(X = x, Z = z, G^c \mid x_{< i})}{P(G^c \mid x_{< i})}
    \left[\privloss_i(x_{\le i}, z) + \log P(G^c \mid x_{< i})\right]
  \end{align*}
  Rearranging gives that
  \begin{align*}
    \E\left[\privloss_i(X_{\le i}, Z) \indic{(X_{\le i}, Z) \not \in G}
      \mid X_{< i} = x_{< i}\right]
    & \le P(G^c \mid x_{< i}) \left[\log |\mc{X}|
      + \log \frac{1}{P(G^c \mid x_{< i})}\right] \\
    & \le P(G^c \mid x_{< i}) \log |\mc{X}|
    + H(\indic{G} \mid X_{< i} = x_{< i}).
  \end{align*}
  Integrating over the marginal of $X_{< i}$ and noting that
  conditioning always reduces entropy, we obtain
  \begin{equation*}
    \E\left[\privloss_i(X_{\le i}, Z) \indic{(X_{\le i}, Z) \not \in G}\right]
    \le P(G^c \mid x_{< i}) \log |\mc{X}|
    + H(\indic{G})
    %% \le P(G^c) \left[\log |\mc{X}|
    %%   + 2 \log \frac{1}{P(G^c)}\right]
  \end{equation*}
  as desired.
  %% whenever $P(G^c) \le \half$, as $p \log \frac{1}{p}
  %% > (1 - p) \log \frac{1}{1 - p}$ for $p < \half$.
\end{proof}

We now turn to the first term in the
expansion~\eqref{eqn:split-info-into-two}. We always have $\privloss_i(X, Z)
\le 6 \diffp$ on the event $G(\eta)$, so
that Lemma~\ref{lemma:control-ungood-part},
coupled with the chain rule~\eqref{eqn:information-as-privacy-loss}
and probability bound~\eqref{eqn:good-set-highprob}
gives
\begin{equation*}
  I(X_{\le n}; Z) \le \sum_{i = 1}^n \left(6 \diffp +
  p_\eta \log |\mc{X}| + h_2(p_\eta)\right)
\end{equation*}
for $p_\eta = \frac{2 \delta}{\eta} + \frac{2 \delta e^\diffp}{
  e^\diffp - 1} + \frac{2 \eta e^{3 \diffp}}{e^{3 \diffp} - 1}$.
This is evidently the first
claim of
Lemma~\ref{lemma:eps-delta-information}.

To see the second claim requires a bit more work,
though the next lemma suffices.
\begin{lemma}
  Let $\eta > 0$ be small enough that
  $\eta (2 e^{6 \diffp} / (e^{3 \diffp} - 1) + 1) \le \half$.
  Then
  \begin{equation*}
    \E[\privloss_i(X_{\le i}, Z) \indic{(X, Z) \in G(\eta)}]
    \le 6 \diffp (e^{6\diffp} - 1)
    + 3 \eta \left[e^{3 \diffp}
      + 3 \eta \frac{e^{12 \diffp}}{(e^{3 \diffp} - 1)^2}\right].
  \end{equation*}
\end{lemma}
\begin{proof}
  Let $(x_{< i}, z) \in E(\eta)$. Then
  \begin{align}
    \lefteqn{\E[\privloss_i(X_{\le i}, Z) \indic{(X, Z) \in G(\eta)}
      \mid Z = z, X_{< i} = x_{< i}]} \nonumber \\
    & = \sum_{x_i : (x_{\le i}, z) \in F_i}
      \privloss_i(x_{\le i}, z) P(X_i = x_i \mid Z = z, x_{< i}) \nonumber \\
    &
    = \sum_{x_i : (x_{\le i}, z) \in F_i}
    \privloss_i(x_{\le i}, z) [P(X_i = x_i \mid Z = z, x_{<i}) - P(X_i = x_i)]
    + \sum_{x_i : (x_{\le i}, z) \in F_i}
    \privloss_i(x_{\le i}, z) P(X_i = x_i) \nonumber \\
    & \le
    6 \diffp (e^{6 \diffp} - 1)
    + \sum_{x_i : (x_{\le i}, z) \in F_i}
    \privloss_i(x_{\le i}, z) P(X_i = x_i)
    \label{eqn:split-privy-losses}
  \end{align}
  where we have used that
  \begin{equation*}
    |P(X_i = x_i \mid Z = z, x_{< i}) - P(X_i = x_i)|
    \le e^{6 \diffp} - 1
  \end{equation*}
  by definition of the set $F_i$ and that $(x_{\le i}, z) \in F_i$,
  and that similarly $|\privloss_i(x_{\le i}, z)| \le 6 \diffp$.

  To bound the second term in the sum~\eqref{eqn:split-privy-losses},
  we note that
  \begin{align*}
    \lefteqn{\sum_{x_i : (x_{\le i}, z) \in F_i}
      \privloss_i(x_{\le i}, z) P(X_i = x_i)} \\
    & = P((X_{\le i}, z) \in F_i \mid x_{< i}) \sum_{x_i : (x_{\le i}, z) \in F_i}
    \privloss_i(x_{\le i}, z) \frac{P(X = x)}{P((X_{\le i}, z) \in F_i \mid x_{< i})} \\
    & \le P((X_{\le i}, z) \in F_i \mid x_{< i})
    \log \frac{P((X_{\le i}, Z) \in F_i \mid x_{< i}, Z = z)}{
      P((X_{\le i}, z) \in F_i \mid x_{< i})}\\
    & = P((X_{\le i}, z) \in F_i \mid x_{< i})
    \log \frac{1 - P((X_{\le i}, Z) \not\in F_i \mid x_{< i}, Z = z)}{
      1 - P((X_{\le i}, z) \not\in F_i \mid x_{< i})}
  \end{align*}
  by Jensen's inequality. Let us bound the logarithmic terms.
  As $(x_{< i}, z) \in E_i(\eta)$ by assumption, we have
  $P((X_{\le i}, z) \not \in F_i \mid x_{< i}) \le e^{3 \diffp}
  P((X_{\le i}, Z) \not \in F_i \mid Z = z, x_{< i}) + \eta$.
  Letting $q = P((X_{\le i}, Z) \not \in F_i \mid Z = z, x_{< i})$ for shorthand,
  Lemma~\ref{lemma:channel-conditions-nicely}
  (Eq.~\eqref{eqn:bounded-set-highprob})
  implies that $q \le \frac{2 \eta e^{3 \diffp}}{e^{3 \diffp} - 1}$,
  and thus
  \begin{equation*}
    \log \frac{1 - P((X_{\le i}, Z) \not \in F_i \mid Z = z, x_{< i})}{1
      - P((X_{\le i}, z) \not \in F_i \mid x_{< i})}
    \le \log \frac{1 - q}{1 - e^{3 \diffp} q - \eta}
    \le (e^{3 \diffp} - 1) q + \eta + (e^{3 \diffp} q + \eta)^2,
  \end{equation*}
  where we have used that $-\log(1 - t) \le t + t^2$ for
  $t \le \half$ and the assumption that
  $e^{3 \diffp} q + \eta < \half$.
  Returning to our bounds on the sum~\eqref{eqn:split-privy-losses},
  we see that
  \begin{equation*}
    \E[\privloss_i(X_{\le i}, Z) \indic{(X, Z) \in G(\eta)}
      \mid z, x_{< i}]
    \le 6 \diffp (e^{6\diffp} - 1)
    + \eta\left[
      2 e^{3 \diffp}
      + 1
      + \eta \left(\frac{2 e^{6 \diffp}}{e^{3 \diffp} - 1} + 1\right)^2
      \right].
  \end{equation*}
  Noting that $e^{6 \diffp} / (e^{3 \diffp} - 1)
  > 5/4$ gives the result.
\end{proof}

\subsection{Proof of Lemma~\ref{lemma:simplify-eps-delta-information}}
\label{sec:proof-simplify-eps-delta-information}

We begin by addressing the
exponential in $\diffp$ terms, which will
allow easier derivation. For all $\diffp \ge 0$, we have
\begin{subequations}
  \label{eqn:dumb-epsilon-bounds}
  \begin{equation}
    \frac{e^{3 \diffp}}{e^{3 \diffp} - 1}
    \le \max\left\{\frac{1}{\diffp}, \frac{e}{e - 1}\right\},
    ~~
    \frac{e^\diffp}{e^\diffp - 1}
    \le \max\left\{\frac{2}{\diffp}, \frac{e}{e - 1}\right\},
  \end{equation}
  and for $\diffp \le \frac{1}{6}$,
  \begin{equation}
    \frac{e^{6 \diffp}}{e^{3 \diffp} - 1}
    \le \frac{3}{4\diffp},
    ~~
    \frac{e^{12 \diffp}}{(e^{3 \diffp} - 1)^2}
    \le \frac{1}{2 \diffp^2},
    ~~
    6 \diffp (e^{6 \diffp} - 1)
    \le 62 \diffp^2.
  \end{equation}
\end{subequations}

Using the bounds~\eqref{eqn:dumb-epsilon-bounds}, we can provide our
desired mutual information bounds. In the case that $\diffp \ge 0$ is
arbitrary, we use the first bound of
Lemma~\ref{lemma:eps-delta-information}.
In this case, to apply the bound it is sufficient that
$p_\eta \le 2 (\frac{\delta}{\eta}
+ \eta \max\{\diffp^{-1}, 2 \}
+ \delta \max\{2 \diffp^{-1}, 2\})
\le \half$,
and taking $\eta = \sqrt{\delta \min\{\diffp, 1/2\}}$ gives that
\begin{equation*}
  p_\eta \le 4 \sqrt{\delta \max\{\diffp^{-1}, 2\}}
  + 2 \delta \max\{\diffp^{-1}, 1\}
  \le 8 \sqrt{\delta \max\{\diffp^{-1}, 1\}} \le \half
\end{equation*}
whenever $\sqrt{\delta \max\{\diffp^{-1}, 1\}} \le 1/16$.
Assuming additionally that
$\sqrt{\delta \max\{\diffp^{-1}, 1\}} \log \frac{1}{\delta
  \max\{\diffp^{-1}, 1\}} \le \diffp$
and $\sqrt{\delta \max\{\diffp^{-1}, 1\}} \log |\mc{X}| \le
\diffp$ gives the first claimed result as
$h_2(p) \le -2p \log p$ for $p \le \half$.

For the second result, under the additional condition that $\diffp \le
1/6$, our chosen $\eta = \sqrt{\delta \min\{\diffp, 1/2\}} =
\sqrt{\delta \diffp}$ satisfies $\eta (\frac{3}{2 \diffp} + 1) \le
\frac{1}{2}$ (as $\delta \le 1 / (64 \diffp)$). When $\delta \le \diffp^3$,
we have
\begin{equation*}
  3 \eta \left[e^{3 \diffp} + 3 \eta \frac{e^{12 \diffp}}{(e^{3\diffp} - 1)^2}
    \right]
  \le 3 \sqrt{\delta \diffp}
  \left(2 + \frac{3}{2} \sqrt{\frac{\delta}{\diffp^3}}\right)
  \le 11 \diffp^2
\end{equation*}
by inequalities~\eqref{eqn:dumb-epsilon-bounds}.
Finally, in this case we again have
$p_\eta \le 8 \sqrt{\delta / \diffp}$, and so
if
\begin{equation*}
  \delta
  \le \frac{\diffp^5}{64 \log^2 |\mc{X}|}
  ~~ \mbox{and} ~~
  \delta \log^2 \frac{\diffp}{\delta} \le \frac{\diffp^5}{16}
\end{equation*}
then $p_\eta\log|\mc{X}| + h_2(p_\eta) \le 2
\diffp^2$. These bounds and
Lemma~\ref{lemma:eps-delta-information} give the second result.

%% \subsection{Proof of Lemma~\ref{lemma:eps-delta-all-information}}
%% \label{sec:proof-eps-delta-all-information}

%% As in the proof of Lemma~\ref{lemma:privacy-information}, we have
%% \begin{equation*}
%%   I(X_{\le n}; Z)
%%   = \sum_{t = 1}^T \sum_{i = 1}^n
%%   I(X_i; Z_i^{(t)} \mid \intoit{Z}).
%% \end{equation*}
%% As each channel $\channel_{Z_i^{(t)}}$ is $(\diffp_{i,t},
%% \delta_{i,t})$-differentially private,
%% we begin by verifying that all the conditions of
%% Lemma~\ref{lemma:simplify-eps-delta-information} hold for
%% each $i, t$.

%% For simplicity let $\delta = \delta_{i,t}$ and $\diffp = \diffp_{i,t}$.
%% If $\diffp \ge 1$, the conditions $\delta \le \frac{1}{64}$ and $\delta
%% \le \diffp^2 / \log^2 |\mc{X}|$ imply each of the preconditions for the
%% bound $I(X_i; Z_i^{(t)} \mid \intoit{Z}) \le 8 \E[\diffp_{i,t}]$
%% in Lemma~\ref{lemma:simplify-eps-delta-information}.
%% Conversely, for $\diffp \le 1$, have $\delta \le \diffp / 64$, $\delta
%% \le \frac{\diffp^5}{64 \log^2 |\mc{X}|}$,
%% and $\frac{\delta}{\diffp} \log^2 \frac{\diffp}{\delta}
%% \le \frac{\diffp^4}{16}$ imply all the conditions sufficient
%% for the second bound
%% $I(X_i; Z_i^{(t)} \mid \intoit{Z}) \le 75 \E[\diffp_{i,t}^2]$
%% (because $75 \diffp^2 \ge 8 \diffp$ for all $\diffp \ge 1/6$).
%% Thus we obtain
%% \begin{equation*}
%%   I(X_{\le n}; Z)
%%   \le \sum_{i = 1}^n \sum_{t = 1}^T
%%   \E[\min\{75 \diffp_{i,t}(\intoit{Z})^2,
%%     8 \diffp_{i,t}(\intoit{Z})\}] \le n \cdot \diffpkl,
%% \end{equation*}
%% as desired.

%% \input{general-dp-info}

\ifcolt

\fi

% -*- mode: latex -*- %

\section{Technical proofs}
\label{sec:deferred-proofs}
  
\subsection{From approximate to pure differential privacy (proof
  of Lemma~\ref{lemma:tv-eps-delta})}
\label{sec:proof-tv-eps-delta}

In this section, we prove Lemma~\ref{lemma:tv-eps-delta}.
The idea in the
lemma is simple (though measurability issues preclude trivial proof): we can
construct alternative channels $\wb{\channel}$ that are close in variation
distance to $\channel$, where $\wb{\channel}$ satisfy pure differential
privacy.

We use Lemma~\ref{lemma:regular-private-densities} along with the
fact that it is no loss of generality to assume that, by approximations
and continuity of $f$-divergences, the $\bZ$ are
discrete~\cite[Thm.~15]{LieseVa06}. Indeed,
the variation distance $\tvnorm{\cdot}$ is an $f$-divergence and
$\P_{\pm 1}$ are marginal distributions over
$\bZ = Z_{\le n}^{(\le T)} \in \mc{Z}^{nT}$. Thus,
letting $\mc{A}$ denote a finite rectangular partition of
$\mc{Z}^{nT}$, meaning that the sets in $A \in \mc{A}$ are of the
form
\begin{equation*}
  A = \prod_{t = 1}^T \left(A_{1,1} \otimes A_{2,1} \otimes \cdots \otimes
  A_{n,1} \right),
  ~~ A_{i,t} \subset \mc{Z},
\end{equation*}
and recalling that rectangles generate the Borel $\sigma$-algebra on
$\mc{Z}^{nT}$, we have the equality~\cite[cf.][Theorem 15]{LieseVa06}
\begin{equation}
  \label{eqn:approx-tv-by-products}
  \tvnorm{\P_{1} - \P_{-1}}
  = \sup_{\mc{A}} \sum_{A \in \mc{A}}
  |\P_{1}(\bZ \in A)
  - \P_{-1}(\bZ \in A)|,
\end{equation}
where the supremum is taken over all finite rectangular partitions of
$\mc{Z}^{nT}$.

We use equality~\eqref{eqn:approx-tv-by-products} to prove the
result. Without loss of generality, we assume the
supremum~\eqref{eqn:approx-tv-by-products} is attained (otherwise, we
simply approximate). As the
partition $\mc{A}$ is finite and consists of rectangular sets, we can
assume the communicated $Z_i^{(t)}$ are discrete. We then have the
following lemma, whose proof we defer to
Section~\ref{sec:proof-regular-private-densities}. This is an extension of
the result~\cite{DworkRoVa10} that $(\diffp,\delta)$-private
channels are close to $(\diffp,0)$-private channels; naive application
of earlier constructions can yield in non-measurable objects and
non-regular conditional probabilities.
\begin{lemma}
  \label{lemma:regular-private-densities}
  Assume that $\mc{Z}$ is countable and
  that for each $i, t \in \N$, the channel
  $\channel(\cdot \mid x_i, \intoit{z})$ is a regular conditional
  probability and that it is $(\diffp, \delta)$-differentially
  private. Then there exists a regular conditional probability
  $\wb{\channel}(\cdot \mid x_i, \intoit{z})$ such that
  $\wb{\channel}$ is $\diffp$-differentially private and
  \begin{equation*}
    \sup_{x_i \in \mc{X}}
    \tvnorm{\channel(\cdot \mid x_i, \intoit{z})
      - \wb{\channel}(\cdot \mid x_i, \intoit{z})} \le
    \half \left[\frac{\delta}{1 + e^\diffp}
      + \frac{\delta}{1 + e^\diffp - \delta}\right].
  \end{equation*}
\end{lemma}

Let $\wb{\channel}$ be the channels
Lemma~\ref{lemma:regular-private-densities} guarantees,
and let $\wb{\P}_{\pm 1}$ be the induced marginal distributions on
$\mc{Z}^{nT}$. Then
\begin{equation*}
  \tvnorm{\P_{1} - \P_{-1}}
  \le \tvnorm{\P_{1} - \wb{\P}_{1}}
  + \tvnorm{\wb{\P}_{1} - \wb{\P}_{-1}}
  + \tvnorm{\wb{\P}_{-1} - \P_{-1}}
\end{equation*}
by the triangle inequality.
Letting $\channeldens$ denote the p.m.f.\ of $\channel$,
we bound $\tvnorm{\P_\packval - \wb{\P}_\packval}$ by expanding
\begin{align*}
  \tvnorm{\P_\packval - \wb{\P}_\packval}
  & = \half \sum_{\bz \in \mc{Z}^{nT}}
  \left|
  \int \left(\channeldens(\bz \mid x_{\le n})
  - \wb{\channeldens}(\bz \mid x_{\le n})\right)
  dP_\packval(x_{\le n})
  \right| \\
  & = \half \sum_{\bz \in \mc{Z}^{nT}}
  \left|
  \int \left(\prod_{i,t} \channeldens(z_i^{(t)} \mid x_{\le n},
  \intoit{z}) - \prod_{i,t}\wb{\channeldens}(z_i^{(t)} \mid x_{\le n},
  \intoit{z})\right) dP_\packval(x_{\le n})\right| \\
  & = \half \sum_{\bz \in \mc{Z}^{nT}}
  \left|
  \int \left(\prod_{i,t} \channeldens(z_i^{(t)} \mid x_i,
  \intoit{z}) - \prod_{i,t} \wb{\channeldens}(z_i^{(t)} \mid x_i,
  \intoit{z})\right) \prod_{i \le n} dP_\packval(x_{\le n})\right|
\end{align*}
where we have used that $Z_i^{(t)}$ is conditionally independent
of $X_{\setminus i}$ given $X_i$ and $\intoit{Z}$.
Now, let
$(j, \tau) \prec (i, t)$ indicate the ordering that
either $\tau < t$ or $j < i$ and $\tau = t$
(and similarly $(j, \tau) \succ (i, t)$ means that
$\tau > t$ or $\tau = t$ and $j > i$),
and define the shorthand
\begin{align*}
  \channeldens_{\prec (i, t)}(\bz \mid x_{\le n})
  \defeq
  \prod_{(j, \tau) \prec (i, t)}
  \channeldens(z_j^{(\tau)} \mid x_j, \intoitj{z})
\end{align*}
and similarly for $\wb{\channeldens}$ and $\channeldens_{\succ (i,t)}$.
Using the telescoping identity that
\begin{equation*}
  \prod_i a_i - \prod_i b_i =
  \sum_i \bigg(\prod_{j < i} a_j\bigg) (a_i - b_i)
  \bigg(\prod_{j > i} b_j\bigg)
\end{equation*}
and the triangle inequality,
we have
\begin{align}
  \label{eqn:finally-get-the-tv}
  \lefteqn{2 \tvnorm{\P_\packval - \wb{\P}_\packval}} \\
  & \le
  \sum_{\stackrel{\packval : \packval_j = 1}{i, t}}
  \int_{\mc{X}^n}
  \underbrace{
    \sum_{\bz \in \mc{Z}^{nT}} \channeldens_{\prec (i,t)}(\bz \mid x_{\le n})
    \left|
    \channeldens(z_i^{(t)} \mid x_i, \intoit{z})
    - \wb{\channeldens}(z_i^{(t)} \mid x_i, \intoit{z})
    \right|
    \wb{\channeldens}_{\succ (i,t)} (\bz \mid x_{\le n})}_{\eqdef
    T_{it}}
  dP_\packval(x_{\le n}).
  \nonumber
\end{align}
The term $T_{it}$ satisfies
\begin{equation*}
  T_{it}
  = \sum_{\stackrel{(j,\tau) \prec (i,t),}{z_j^{(\tau)}}}
  \channeldens_{\prec (i,t)}(\bz \mid x_{\le n})
  \sum_{z_i^{(t)} \in \mc{Z}}
  \left|\channeldens(z_i^{(t)} \mid x_i, \intoit{z})
  - \wb{\channeldens}(z_i^{(t)} \mid x_i, \intoit{z})\right|
  \sum_{\stackrel{(j,\tau) \succ (i,t),}{z_j^{(\tau)}}}
  \wb{\channeldens}_{\succ (i,t)} (\bz \mid x_{\le n}),
\end{equation*}
where the variation distance guarantee of
Lemma~\ref{lemma:regular-private-densities} (coupled with the privacy
Assumption~\ref{assumption:summed-approximate-privacy}) guarantees that
\begin{align*}
  \sum_{z_i^{(t)} \in \mc{Z}}
  \left|\channeldens(z_i^{(t)} \mid x_i, \intoit{z})
  - \wb{\channeldens}(z_i^{(t)} \mid x_i, \intoit{z})\right|
  & \le
  \half \left[\frac{\delta_{i,t}(\intoit{z})}{
      1 + e^{\diffp_{i,t}(\intoit{z})}}
    + \frac{\delta_{i,t}(\intoit{z})}{
      1 + e^{\diffp_{i,t}(\intoit{z})} - \delta_{i,t}
      (\intoit{z})}
    \right] \\
  & \le \delta_{i,t}(\intoit{z})
\end{align*}
as $\delta_{i,t} \in [0, 1]$.
We thus obtain
\begin{align*}
  T_{it}
  & \le \sum_{\stackrel{(j,\tau) \prec (i,t),}{z_j^{(\tau)}}}
  \channeldens_{\prec (i,t)}(\bz \mid x_{\le n})
  \delta_{i,t}(\intoit{z})
  \max_{z_i^{(t)} \in \mc{Z}}
  \sum_{\stackrel{(j,\tau) \succ (i,t),}{z_j^{(\tau)}}}
  \wb{\channeldens}_{\succ (i,t)} (\bz \mid x_{\le n}) \\
  & = \sum_{\stackrel{(j,\tau) \prec (i,t),}{z_j^{(\tau)}}}
  \channeldens_{\prec (i,t)}(\bz \mid x_{\le n})
  \delta_{i,t}(\intoit{z})
  = \E_\channel\left[\delta_{i,t}(\intoit{Z}) \mid X_{\le n} = x_{\le n}\right],
\end{align*}
where the equality follows because p.m.f.s sum to 1.
Substituting this into inequality~\eqref{eqn:finally-get-the-tv}
yields
\begin{equation*}
  \tvnorm{\P_\packval - \wb{\P}_\packval}
  \le \half 
  \sum_{i,t} \E_{P_\packval}[\delta_{i,t}(\intoit{Z})]
  = \half \sum_{i,t} \E_{\P_\packval}\left[\delta_{i,t}(\intoit{Z})\right]
  \le \frac{\totaldelta}{2},
\end{equation*}
the final inequality following again by
Assumption~\ref{assumption:summed-approximate-privacy}. This gives 
Lemma~\ref{lemma:tv-eps-delta}.

\subsection{Proof of Lemma~\ref{lemma:regular-private-densities}}
\label{sec:proof-regular-private-densities}

If the space $\mc{X}$ is countable, then this result is essentially due to
\citet{DworkRoVa10} (see Lemma~2.1 in the long version of their paper)
once we apply the averaging technique in the end of this proof.
When the space $\mc{X}$ is not countable, we must be more careful to
maintain measurability, so that our construction actually yields a valid
channel. Because $\mc{Z}$ is countable, however, it is possible to achieve
our desired result.  Without loss of generality, because $\mc{Z}$ is
countable, we may assume that $\channel$ has a density (p.m.f.) $q$ on
$\mc{Z}$, as each $\channel(\cdot \mid x_i, \intoit{z})$ is absolutely
continuous w.r.t.\ the counting measure on $\mc{Z}$.

Let us take $x, x' \in \mc{X}$ otherwise arbitrary, and let $w = \intoit{z}$
for shorthand, so that we have densities
$q(z \mid x, w)$ and $q(z \mid x', w)$, both of which are
measurable in their (three) arguments. Then define the two sets
\begin{equation*}
  S_{x} \defeq \{z \in \mc{Z} \mid q(z \mid x, w) > e^\diffp q(z \mid x', w)
  \}
  ~~\mbox{and} ~~
  S_{x'} \defeq \{z \in \mc{Z} \mid q(z \mid x', w) > e^\diffp q(z \mid x, w)
  \}
\end{equation*}
and the intermediate densities
\begin{align*}
  q_1(z \mid x; x', w)
  \defeq \, &
  \left[q(z \mid x, w) + q(z \mid x', w)\right]
  \left(\frac{e^\diffp}{e^\diffp + 1} \indic{z \in S_x}
  + \frac{1}{e^\diffp + 1} \indic{z \in S_{x'}}\right) \\
  % \indic{q(z \mid x, w) > e^\diffp q(z \mid x', w)} \\
  %% + \frac{1}{1 + e^\diffp}
  %% \left[q(z \mid x, w) + q(z \mid x', w)\right]
  %% \indic{q(z \mid x, w) < e^{-\diffp} q(z \mid x', w)} \\
  & + q(z \mid x, w) \indic{z \not \in S_x \cup S_{x'}}, \\
  q_1(z \mid x'; x, w)
  \defeq \, &
  \left[q(z \mid x, w) + q(z \mid x', w)\right]
  \left(\frac{1}{e^\diffp + 1} \indic{z \in S_x}
  + \frac{e^\diffp}{e^\diffp + 1} \indic{z \in S_{x'}}\right) \\
  & + q(z \mid x, w) \indic{z \not \in S_x \cup S_{x'}}.
\end{align*}
Evidently these quantities satisfy
\begin{equation*}
  e^{-\diffp}
  \le \frac{q_1(z \mid x; x', w)}{q_1(z \mid x'; x, w)}
  \le e^\diffp
\end{equation*}
for all $z \in \mc{X}$, and moreover, by inspection they are
$(z, x, x', w)$-measurable as they are the product of measurable functions.
Let $\channel_1$ denote the induced measure (not necessarily probabilities)
on $\mc{Z}$ by the constructed $q_1$.

With this definition of $Q_1$, we may define the two quantities
\begin{align*}
  \alpha_x
  \defeq
  \channel(S_x \mid x, w)
  - \channel_1(S_x \mid x; x', w)
  & = \channel(S_x \mid x, w)
  - \frac{e^\diffp}{1 + e^\diffp}
  (\channel(S_x \mid x, w) + \channel(S_x \mid x', w)) \\
  & = \frac{\channel(S_x \mid x, w) - e^\diffp \channel(S_x \mid x', w)}{
    1 + e^\diffp} \in \left[0, \frac{\delta}{1 + e^\diffp}\right]
\end{align*}
and similarly
\begin{equation*}
  \alpha_{x'}
  \defeq \channel(S_{x'} \mid x', w) - \channel_1(S_{x'} \mid x'; x, w)
  \in \left[0, \frac{\delta}{1 + e^\diffp}\right].
\end{equation*}
We also have $\channel(S_x \mid x, w)
- \channel_1(S_x \mid x; x', w)
= \channel_1(S_x \mid x'; x, w) - \channel(S_x \mid x', w)$
and $\channel(S_{x'} \mid x', w)
- \channel_1(S_{x'} \mid x'; x, w)
= \channel_1(S_{x'} \mid x; x', w) - \channel(S_{x'} \mid x, w)$
by construction.
With these definitions and equalities, we have the variation bound
\begin{align*}
  \lefteqn{\tvnorm{\channel(\cdot \mid x, w)
      - \channel_1(\cdot \mid x'; x, w)}} \\
  & = \half \left(\channel(S_x \mid x, w) - \channel_1(S_x \mid x; x', w)\right)
  + \half \left(\channel_1(S_{x'} \mid x; x', w)
  - \channel(S_{x'} \mid x, w)\right) \\
  & = \half \alpha_x + \half \alpha_{x'} \le \frac{\delta}{1 + e^\diffp}.
\end{align*}

The normalized densities
\begin{equation*}
  q_0(z \mid x; x', w) \defeq \frac{q_1(z \mid x; x', w)}{
    \sum_z q_1(z \mid x; x', w)}
  ~~ \mbox{and} ~~
  q_0(z \mid x'; x, w) \defeq \frac{q_1(z \mid x'; x, w)}{
    \sum_z q_1(z \mid x; x', w)}
\end{equation*}
are both $(z, x, x', w)$-measurable, and they satisfy
the ratio guarantee
$|\log \frac{q_0(z \mid x; x', w)}{q_0(z \mid x'; x, w)}| \le \diffp$.
Moreover, we have
$Q_1(\mc{Z} \mid x; x', w) = 1 - \alpha_x + \alpha_{x'}$ and
$Q_1(\mc{Z} \mid x'; x, w) = 1 - \alpha_{x'} + \alpha_x$.
We then have
\begin{align*}
  \lefteqn{\tvnorm{Q(\cdot \mid x, w)
      - Q_0(\cdot \mid x; x', w)}} \\
  & \le \tvnorm{Q(\cdot \mid x, w) - Q_1(\cdot \mid x; x', w)}
  + \tvnorm{Q_1(\cdot \mid x; x', w) - Q_0(\cdot \mid x; x', w)} \\
  & = \frac{\alpha_x + \alpha_{x'}}{2}
  + \half \left|\frac{1}{Q_1(\mc{Z} \mid x; x', w)} - 1 \right|
  = \frac{\alpha_x + \alpha_{x'}}{2}
  + \frac{|\alpha_x - \alpha_{x'}| / 2}{1 - \alpha_x + \alpha_{x'}}
  \le \half \left[\frac{\delta}{1 + e^\diffp}
    + \frac{\delta}{1 + e^\diffp - \delta}\right].
\end{align*}
where we have taken $\alpha_x = \delta / (1 + e^\diffp)$ and $\alpha_{x'} =
0$ to maximize the sum above. An identical bound holds on $\tvnorms{Q(\cdot
  \mid x', w) - Q_0(\cdot \mid x'; x, w)}$.

It remains to construct our desired regular conditional distribution
$\wb{Q}$. To that end, note that
each of $q_0(z \mid x; x', w)$ and $q_0(z \mid x'; x, w)$ are
measurable in $(z, x, x', w)$ by our construction. Choosing an
arbitrary probability measure $\lambda$ on the space $\mc{X}$, we may
then define
\begin{equation*}
  \wb{q}(z \mid x, w) \defeq \int q_0(z \mid x; x', w) d\lambda(x')
\end{equation*}
for all $z, x, w$.
Taking $\wb{Q}$ to be the associated probability measure, we evidently
have that $\wb{Q}$ is a regular conditional probability, that
$\tvnorms{\wb{Q}(\cdot \mid x, w) - Q(\cdot \mid x, w)}
\le \half(\frac{\delta}{1 + e^\diffp} + \frac{\delta}{1 + e^\diffp - \delta})$,
and that
$e^{-\diffp} \le \wb{q}(z \mid x, w) / \wb{q}(z \mid x', w)
\le e^\diffp$ as desired.

%% \section{Deferred proofs}

\subsection{Proof of Lemma~\ref{lemma:gaussian-achievability}}
\label{sec:proof-gaussian-achievability}

We allow $c, C$ to be numerical constants whose value may change from
line to line. We also assume $\sigma^2 > 0$ is at least a numerical constant.
First, we have that $|Z_i| \le b$. Thus
\begin{equation*}
  \P(|\wb{Z}_n - \E[\wb{Z}_n]| \ge t) \le \exp\left(-\frac{n
    t^2}{2 b^2}\right)
  ~~ \mbox{for~} t \ge 0
\end{equation*}
by Hoeffding's inequality.
Note that $\E[\wb{Z}_n] \in [1 - 2 \Phi(1/\sigma), 1 + 2 \Phi(-1/\sigma)]
\subset [e^{-c / \sigma^2}, 1 - e^{-c / \sigma^2}] = [e^{-C}, 1 - e^{-C}]$
by our assumption that $\sigma$ is at least a constant.
Now, let $\mc{E}$ denote the event that $\wb{Z}_n \in [e^{-c / \sigma^2} /
  2, 1 - e^{-c / \sigma^2}/2]$, which happens with probability at least $1 -
\exp(-c n / b^2)$.  On this event, a Taylor expansion of
$\Phi^{-1}$ gives
\begin{align*}
  \sigma \Phi^{-1}
  \left(\frac{1 - \wb{Z}_n}{2}\right)
  & = \sigma \Phi^{-1}
  \left(\frac{1 - \E[\wb{Z}_n]}{2}\right)
  + \sigma \frac{\wb{Z}_n - \E[\wb{Z}_n]}{
    \phi(\theta)}
  \pm C \sigma (\wb{Z}_n - \E[\wb{Z}_n])^2 \\
  & = \theta
  + \sigma \frac{\wb{Z}_n - \E[\wb{Z}_n]}{
    \phi(\theta)}
  \pm C \sigma (\wb{Z}_n - \E[\wb{Z}_n])^2.
\end{align*}
We have $|\wb{Z}_n - \E[\wb{Z}_n]| \le
\sqrt{2 b^2 t / n}$ with probability
at least $1 - e^{-t}$
by Hoeffding's inequality,
and we also have
\begin{align*}
  \E_\theta[\ltwos{\what{\theta}_n - \theta}^2]
  & \le \frac{2 \sigma^2}{\phi(\theta)^2}
  \E[(\wb{Z}_n - \E[\wb{Z}_n])^2]
  + C^2 \sigma^2 \E[(\wb{Z}_n - \E[\wb{Z}_n])^4]
  + C \P(\mc{E}^c) \\
  & \le C \frac{b^2 \sigma^2}{n}
  + C  \frac{b^4 \sigma^2}{n^2}
  + C e^{-c n / b^2},
\end{align*}
where the second inequality follows by the $b$-boundendess
of the $Z_i$ and standard moment bounds for sub-Gaussian random
variables~\citep{Vershynin12}.

\end{document}